\documentclass[reqno, 11pt]{amsart}

\usepackage[]{hyperref}

\usepackage{bbm}
\usepackage{url}
\usepackage{stmaryrd}
\usepackage{mathrsfs}
\usepackage{cases}
\usepackage{amsfonts}
\usepackage{amsmath,amssymb,bm}
\usepackage{enumerate}
\usepackage{tikz}
\usepackage{extarrows}

\usepackage{citeref}

\allowdisplaybreaks[4]
\newtheorem{theorem}{Theorem}[section]
\newtheorem{lemma}[theorem]{Lemma}
\newtheorem{proposition}[theorem]{Proposition}

\theoremstyle{definition}
\newtheorem{definition}[theorem]{Definition}

\newtheorem{remark}[theorem]{Remark}
\numberwithin{equation}{section}

\let\a=\alpha \let\al=\alpha
\let\b=\beta

\let\g=\gamma
\let\G=\Gamma 
\let\i=\infty 
\let\fy=\infty

\let\r=\rho

\let\ep=\epsilon
\let\la=\lambda
\let\La=\Lambda

\let\Om=\Omega

\let\th=\theta

\let\wt=\widetilde
\let\va=\varphi

\let\vph=\varphi

\def\bbR{\mathbb{R}}

\def\bbZ{\mathbb{Z}}  \def\zn{\mathbb{Z}^n}
\def\bbN{\mathbb{N}}

\def\scrF{\mathscr{F}}

\def\calC{\mathcal {C}}

\def\calF{\mathcal {F}}
\def\calG{\mathcal {G}}

\def\calK{\mathcal {K}}
\def\calL{\mathcal {L}}
\def\calM{\mathcal {M}}

\def\calS{\mathcal {S}}

\newcommand{\be}{\begin{equation*}} 	\newcommand{\ee}{\end{equation*}}
\newcommand{\ben}{\begin{equation}} 	\newcommand{\een}{\end{equation}}
\newcommand{\bn}{\begin{enumerate}} 	\newcommand{\en}{\end{enumerate}}
\newcommand{\bs}{\backslash}
\newcommand{\lan}{\langle}          \newcommand{\ran}{\rangle}
\def\mp{M^p}    \def\mpd{M^p(\rd)}    
    \def\mqd{M^q(\rd)}
\def\mpq{M^{p,q}}   \def\mpqd{M^{p,q}(\rd)}

          \def\mid{M^{1}(\rd)}
     \def\mfd{M^{\infty}(\rd)}

  \def\lpd{L^p(\rd)}

\def\bka{\Box_k^{\al}}

         \def\mpad{M^p_{\al}(\rd)}
         \def\mqad{M^q_{\al}(\rd)}
    \def\mpqad{M^{p,q}_{\al}(\rd)}

\def\zd{{{\mathbb{Z}}^d}}   \def\zdd{{{\mathbb{Z}}^{2d}}}
\def\zn{{{\mathbb{Z}}^n}}   

\def\rr{{\mathbb R}}
\def\rd{{{\rr}^d}}          \def\rdd{{{\rr}^{2d}}}
          
\begin{document}
\title[Boundedness and compactness kernel theorem for $\alpha$-modulation spaces]{Boundedness and compactness kernel theorem for $\alpha$-modulation spaces}
\author{GUOPING ZHAO}
\address{School of Mathematics and Statistics, Xiamen University of Technology, Xiamen, 361024, P.R.China}
\email{guopingzhaomath@gmail.com}
\author{WEICHAO GUO$^*$}
\address{School of Science, Jimei University, Xiamen, 361021, P.R.China}
\email{weichaoguomath@gmail.com}
\thanks{$^*$The corresponding author}
\subjclass[2010]{42B35,42C15,47G30}
\keywords{Kernel theorem, $\alpha$-modulation space, boundedness, compactness}

\begin{abstract}
  This paper is devoted to establishing the kernel theorems for $\al$-modulation spaces
in terms of boundedness and compactness.
We characterize the boundedness of a linear operator $A$ from an $\al$-modulation space $\mpad$
into another $\al$-modulation space $\mqad$, by the membership of its distributional kernel in mixed $\al$-modulation spaces. 
We also characterize the compactness of $A$ by means of the kernel in a certain closed subspace of mixed $\al$-modulation spaces. 
The proofs are based on the viewpoint that the action of the linear operator on certain function space can be reduced to the action on the suitable atoms of this function space.
\end{abstract}

\maketitle

\section{Introduction}
The famous Schwartz kernel theorem \cite[Th.5.2.1]{Hoermander1983} states that the tempered distributions $K\in\calS'(\rdd)$ can correspond one-to-one with the linear operators 
$A\in \calL(\calS(\rd), \calS'(\rd))$ by
\ben\label{intro-1}
 \langle Af, g \rangle
 =
 \langle K, g\otimes\bar{f} \rangle,\ \ \ \  f,g\in \calS(\rd),
\een
where $K$ is known as the distributional kernel of $A$. Formally, we can write
\be
Af(x)=\int_{\rd}K(x,y)f(y)dy,
\ee
where the integral makes sense for nice distributional kerenl $K$ and test function $f$.
This corresponding relation is one of the most important fundamental tools in the modern functional analysis.
In the field of time-frequency analysis, Feichtinger's kernel theorem was announced in \cite{Feichtinger1980CRHdSdldSSA} and proved in \cite{FeichtingerGroChenig1997JoFA}.
It gives a one-to-one correspondence as  \eqref{intro-1}
between linear operators on $\calL(M^1(\rd), M^{\fy}(\rd))$ and the elements in modulation space $M^{\fy}(\rdd)$.
Note that in Feichtinger's kernel theorem, both the test function space $M^1(\rd)$ and the distribution space $M^{\fy}(\rd)=(M^1(\rd))^*$ are Banach spaces.
Therefore, this theorem has technical superiority over the classical Schwartz kernel theorem in dealing with problems in the field of time-frequency analysis.

Modulation space is a class of important function spaces in the field of time-frequency analysis.
It was invented by Feichtinger \cite{Feichtinger1983} in 1983.
The (unweighted) modulation space $\mpqd$ is defined by measuring the decay and integrability of the short time Fourier transform (STFT) as follows:
\be
M^{p,q}(\rd)=\{f\in \calS'(\rd): V_gf\in L^{p,q}(\rdd) \},
\ee
where $V_gf$ stands for the STFT of $f$ associated with the window function $g\in \calS(\rd)$.
For simplicity, we write $M^p(\rd)=M^{p,p}(\rd)$.  
From the known embedding relation $M^1(\rd)\subset \mpqd\subset \mfd$ and the duality relation $M^{\fy}(\rd)=(M^1(\rd))^*$,
the modulation spaces $M^1(\rd)$ and $M^{\fy}(\rd)$ can be regarded as the test function space and the distribution space of the function spaces $\mpq$ in the range of $1\leq p,q\leq \fy$, respectively.
We refer to Section 2 for more details about modulation spaces. We use $\calM^{p,q}(\rd)$ to denote the $\calS(\rd)$ closure in $M^{p,q}(\rd)$.

Beyond the traditional framework of kernel theorems from the test function space $X$ into the distribution space $X^*$, 
Cordero and Nicola \cite{CorderoNicola2019JFAA} 
proved a generalized kernel theorem associated with the linear operators belonging to $\calL(\mid,M^p(\rd))$ and $\calL(M^p(\rd),\mfd)$ respectively.
One can also see \cite{BalazsGroechenigSpeckbacher2019TotAMSSB} for a more abstract version of kernel theorems on several coorbit spaces.

Recall that the classical Besov space $B^{p,q}(\rd)$ is a class of frequency decomposition spaces with dyadic decomposition.
By using an equivalent norm (see \cite{Triebel1983ZFA}), the modulation space can be compared with the Besov space
and can be considered as 
the Besov type space associated with uniform decomposition on the frequency domain.
As an intermediate decomposition between the uniform and dyadic decompositions,
the $\a$-covering was first introduced by Feichtinger \cite{FeichtingerGroebner1985MN, Feichtinger1987MN}.
Then based on the $\alpha$-covering, Gr\"{o}bner \cite{Grobner1992} introduced
the $\alpha $-modulation spaces $M^{p,q}_{\alpha}(\rd)$ for $\alpha \in \lbrack 0,1)$.

Accordingly, the $\a$-modulation space (see the definition in Section 2), generated by the $\a$-covering, can be considered formally as the intermediate space between modulation and Besov spaces.
More precisely,
the space $\mpqad$
coincides with the modulation space $\mpq$ when $\alpha=0$, and
the Besov space $\ B^{p,q}(\rd)$ can be regarded as the limit case of
$\mpqad$ as $\alpha\rightarrow 1 $ (see \cite{Grobner1992}). For the sake of convenience, we
can view the Besov space as a special $\alpha$-modulation space.

As mentioned above, the $\al$-modulation space is a class of more generalized function spaces that bridges the modulation and Besov spaces.
A natural question is whether the kernel theorem can be established in the framework of $\al$-modulation spaces, and if so, what form this kernel theorem should be.
In order to examine this issue, we first review the main method used in the existing kernel theorems for modulation spaces.

The main idea of the previous results is to establish an equivalent relation by  using the
(continuous or discrete) short-time Fourier transform, between the boundedness of the target operators $A$ from the function spaces $X$ into $Y$ and
that of its corresponding operators $\wt{A}$ from the corresponding function spaces $\wt{X}$ into $\wt{Y}$, realizing the transformation of the desired kernel theorem to the corresponding
``known kernel theorem" on (discrete) Lebesgue spaces.

Obviously, at least on the surface, the above idea is not valid in the framework of $\al$-modulation space, which lacks the corresponding tools such as the short time Fourier transform and its inverse transform.
However, we point out that the essence that allows to establish the corresponding kernel theorem for $A\in\calL(X,Y)$ is actually hidden in the following properties:
\bn
\item Every element in $X$ can be expanded by a series consisting of ``suitable atoms" multiplied by corresponding coefficients;
\item The action on $X$ of $A$ can be characterized by the actions on the atoms;
\item For linear operator $A$ and atom $x\in X$, the $Y$-norm of $Ax$ can be represented as a quantity related to the distributional kernel of $A$.
\en

Based on the above observation,
the boundedness of linear operators on $\al$-modulation spaces can be
characterized in terms of the membership of their distributional kernels to certain mixed $\al$-modulation spaces.
We refer to Section 5 for the definitions of mixed $\al$-modulation space $M^{p_1,p_2,q_1,q_2}_{\al,s,t}(c_i)$, $i=1,2$.

For two $\al$-modulation spaces $M^p_{\al}(\rd)$ and $M^q_{\al}(\rd)$, we use $A\in\calL(M^p_{\al}(\rd), M^q_{\al}(\rd))$ to denote that $A\in \calL(\calS(\rd), \calS'(\rd))$ first defined on  the Schwartz space, satisfying $\|Af\|_{\mqad}\lesssim \|f\|_{\mpad}$ for all $f\in \calS(\rd)$, and then can be extended to a bounded operator from $\calM^p_{\al}(\rd)$ to $M^q_{\al}(\rd)$,
where $\calM^p_{\al}(\rd)$ denotes the $\calS(\rd)$ closure in $\mpad$.
See Section 2 for the precise definition of $\al$-modulation space. Denote $q\wedge 1:=\min\{q,1\}$.
The translation operator is defined by $T_xf(t)=f(x-t)$.
We state our main theorems as follows.

\begin{theorem}\label{thm-sBKa}
	Suppose that $p,q\in (0,\fy]$ satisfying $p\leq q\wedge 1$.
	Let $A\in \calL(\calS(\rd), \calS'(\rd))$ be a linear operator with
	distributional kernel $K_A\in \calS'(\rdd)$.		The following statements are equivalent
	\bn
	\item
	$A\in\calL(M^p_{\al}(\rd), M^q_{\al}(\rd))$;
	\item
	$\{\lan k\ran^{\frac{\al d}{1-\al}(1/p-1)}A(T_{y} \check{\eta_k^{\al}}): k\in \zd, y\in \rd\}$
	is a bounded subset of $M^{q}_{\al}(\rd)$;
	\item
	$K_A\in M^{q,q,\fy,\fy}_{\al,0,\al d(1/p-1)}(c_1)$.
	\en
	Moreover, we have the norm estimate
	\be
	\begin{split}
		\|A\|_{\calL(\mpad,\mqad)}
		\sim
		\sup_{k\in \zd, y\in \rd}\lan k\ran^{\frac{\al d}{1-\al}(1/p-1)}\|A(T_{y} \check{\eta_k^{\al}})\|_{M^q_{\al}}
		\sim 
		\|K_A\|_{M^{q,q,\fy,\fy}_{\al,0,\al d(1/p-1)}(c_1)},
	\end{split}
	\ee
	where the implicit constant is independent of $A$.
	\end{theorem}
We have the corresponding dual conclusion for $A\in\calL(M^p_{\al}(\rd), M^{\fy}_{\al}(\rd))$.
\begin{theorem}\label{thm-sBKa-dual}
	Suppose that $p\in [1, \i]$.
	Let $A\in \calL(\calS(\rd), \calS'(\rd))$ be a linear operator with
	distributional kernel $K_A\in \calS'(\rdd)$.	
We have
 \be
A\in\calL(M^p_{\al}(\rd), M^{\fy}_{\al}(\rd))\Longleftrightarrow K_A\in M^{p',p',\fy,\fy}_{\al,0,0}(c_2)
 \ee
 Moreover, we have the norm estimate
 \be
    \begin{split}
    \|A\|_{\calL(M^p_{\al}(\rd), M^{\i}_{\al}(\rd))}
    \sim 
    \|K_A\|_{M^{p',p',\fy,\fy}_{\al,0,0}(c_2)}.
    \end{split}
    \ee
    where the implicit constant is independent of $A$.
\end{theorem}

\begin{remark}
Note that the case $A\in\calL(M^p_{\al}(\rd), M^{\i}_{\al}(\rd))$ with $p\leq 1$ has been dealt with in Theorem \ref{thm-sBKa}. Combining this with Theorem \ref{thm-sBKa-dual},
we actually established the kernel theorem of $A\in\calL(M^p_{\al}(\rd), M^{\i}_{\al}(\rd))$ for all $p\in (0,\fy]$.
\end{remark}

Again, by making use the strategy of transforming the action of the linear operator on the function space into its action on the atoms, the compactness of linear operators on $\al$-modulation spaces can be
characterized in terms of the membership of their distributional kernels to certain closed subspace of mixed $\al$-modulation spaces.
See Section 5 for the definition of $\tilde{M}^{q,q,\fy,\fy}_{\al,0,\al d(1/p-1)}(c_1)$.
\begin{theorem}\label{thm-sCKa}
	Suppose that $p,q\in (0,\fy)$ satisfying $p\leq q\wedge 1$.
	Let $A\in \calL(\calS(\rd), \calS'(\rd))$ be a linear operator with
	distributional kernel $K_A\in \calS'(\rdd)$.	The following statements are equivalent
	\bn
	\item
	$A\in\calK(M^p_{\al}(\rd), M^q_{\al}(\rd))$;
	\item
	$\{\lan k\ran^{\frac{\al d}{1-\al}(1/p-1)}A(T_{y} \check{\eta_k^{\al}}): k\in \zd, y\in \rd\}$
	is a totally bounded subset of $M^{q}_{\al}(\rd)$;
	\item
	$K_A\in \tilde{M}^{q,q,\fy,\fy}_{\al,0,\al d(1/p-1)}(c_1)$.
	\en
	\end{theorem}

 \begin{remark}
     As mentioned above, the proofs of our kernel theorems are based on the expansion of the elements in $\al$-modulation spaces.
     However, such an expansion is not unique, and this also leads to the diversity of the descriptions of the theorems. 
     We put the theorems (Theorems \ref{thm-BKa}, \ref{thm-BKa-dual} and \ref{thm-CKa}) that is more closed related to the general expansion in Sections 3 and 4, and put their concise form (Theorems \ref{thm-sBKa}, \ref{thm-sBKa-dual} and \ref{thm-sCKa}) here.
     Note that although the descriptions of these theorems may be different, in essence they are equivalent and all are
     the characterizations of the bounded and compact linear operators on the $\al$-modulation spaces.
     From another point of view, the properties of the linear operator are characterized by the mixed $\al$-modulation spaces,
     and these mixed spaces have several equivalent norms just like the case of $\al$-modulation spaces. 
     See Section 5 for more details.
 \end{remark}

 \begin{remark}
     Recall that the Besov space can be regarded as the limiting case of $\al$-modulation space as $\al\rightarrow 1$.
Using the similar argument, one can also establish the corresponding boundedness and compactness kernel theorems for Besov spaces.
The more general case such as $A\in \calL(M^{p}_{s,\al_1}(\rd),M^{q}_{t,\al_2}(\rr^n))$ is also expected to be characterized.
 \end{remark}

The rest of this paper is organized as follows.
In Section 2, we collect some basic concepts and
properties used in this paper.
Section 3 is devoted to the boundedness kernel theorem for $\al$-modulation spaces, i.e., Theorems \ref{thm-BKa} and \ref{thm-BKa-dual}.
To achieve our goal, we establish some propositions which are used to decompose functions into specific series.
We also establish some auxiliary conclusions of independent interest for the duality argument used in Theorem \ref{thm-BKa-dual}.
The compactness kernel theorem for $\al$-modulation spaces (Theorem \ref{thm-CKa}) will be proved in Section 4.
For this purpose, we establish an expansion for the elements in $\al$-modulation spaces.
We also give some useful criteria for the totally bounded subsets in function spaces.
In Section 5, we discuss the equivalent forms of our main theorems, and also provide the proofs of 
their concise forms, i.e., Theorems \ref{thm-sBKa}, \ref{thm-sBKa-dual} and \ref{thm-sCKa}.
We also revisit the case of modulation space in this section.

\section{Preliminary}
First, we recall some notations. 
Let $d$ be the dimension of the Euclidean space.
Let $C$ be a positive
constant that may depend on $d,p,q,\al$. We use $A\lesssim B$ to denote the statement that $A\leq CB$, and use
$A\sim B$ to denote the statement $A\lesssim B\lesssim A$.
The notation $\mathscr{L}$ is used to denote some large positive number which may be changed
corresponding to the exact environment.
The Schwartz function space is denoted by $\calS(\rd)$, the space of tempered distributions by $\calS'(\rd)$.
We use the brackets $\lan f,g\ran$ to denote the extension to $\calS'(\rd)\times \calS(\rd)$ from the usual inner product on $L^2(\rd)$.
The Fourier transform $\mathscr {F}f$ and the inverse Fourier transform $\mathscr {F}^{-1}f$  of $f\in \calS(\rd)$ is defined by
$$
\mathscr {F}f(\xi)=\hat{f}(\xi)=\int_{\mathbb{R}^{d}}f(x)e^{-2\pi ix\cdot \xi}dx, \ \ \ 
\mathscr {F}^{-1}f(x)=\check{f}(x)=\int_{\mathbb{R}^{d}}f(\xi)e^{2\pi ix\cdot \xi}d\xi.
$$

Next, we introduce the partition of unity associated with $\al\in [0,1)$.
Let $\rho_0$ be a smooth radial bump function satisfying $\rho_0(\xi)=1$ for $|\xi|\leq 1/4$, and $\text{supp}\rho_0=B(0,1/2)$.
For any $k\in\mathbb{Z}^d$, we set
\ben\label{pre-1}
\rho_k^{\alpha}(\xi)=\rho\left( \frac{\xi-\langle
	k\rangle^{\frac{\alpha}{1-\alpha}}k}{\langle
	k\rangle^{\frac{\alpha}{1-\alpha}} }\right),\ \ \ \ \r(\xi)=\r_0(\frac{\xi}{2C}).
\een
There exists a large constant $C$ such that
$
\sum_{l\in\zd}\rho_l^{\alpha}(\xi)\gtrsim 1
$
for all $\xi\in \rd$.
Denote by
\ben\label{pre-2}
	\eta_k^{\alpha}(\xi)=\rho_k^{\alpha}(\xi)\bigg(\sum_{l\in\zd}\rho_l^{\alpha}(\xi)\bigg)^{-1}.
\een
In our definition, we find that $\text{supp}\eta_k^{\al}=B(\langle k\rangle^{\frac{\alpha}{1-\alpha}}k, C\lan k\ran^{\frac{\al}{1-\al}})$ for all $k\in \zd$.
The sequence $\{\eta_{k}^{\alpha}\}_{k\in\mathbb{Z}^{d}}$ constitutes a smooth partition of unity of $\mathbb{R}^{d}$.
The corresponding frequency decomposition operators are defined by
\begin{equation}
	\Box_{k}^{\alpha}:= \scrF^{-1}\eta_{k}^{\alpha}\scrF, \, k\in \mathbb{Z}^{d}.
\end{equation}
Let $0< p,q \leq \infty$, $s\in \rr$, $\alpha \in [0,1)$. The $\alpha$-modulation space associated with above decomposition is defined by
\begin{equation}
	M^{p,q}_{s,\al}(\rd)
	=
	\bigg\{
	f\in \calS(\rd):
	\|f\|_{M^{p,q}_{s,\al}(\rd)}
	=\bigg( 
          \sum_{k\in \zd}\|\Box_k^{\alpha} f\|_{L^p}^{q} \lan k \ran^{\frac{sq}{1-\a}}
     \bigg)^{1/q}<\infty
	\bigg\},
\end{equation}
with the usual modifications when $q=\infty$. For simplicity, we write
$M^{p,q}_{\al}(\rd)=M^{p,q}_{0,\al}(\rd)$, $M^{p}_{s,\al}(\rd)=M^{p,p}_{s,\al}(\rd)$ and 
$M^{p}_{\al}(\rd)=M^{p}_{0,\al}(\rd)$.
Note that the modulation space coincides with $\a$-modulation space when $\a=0$.
\begin{remark}
	In the definition of $\a$-modulation spaces, the sequence of $\{\eta_{k}^{\alpha}\}_{k\in\mathbb{Z}^{d}}$
	can be replaced by a large class of suitable sequences associated with $\al$ satisfying that
 \ben\label{smooth-partition-alpha}
\begin{cases}
|\eta _{k}^{\alpha }(\xi )|\gtrsim 1,~\text{if}~|\xi -\langle k\rangle ^{\frac{%
\alpha }{1-\alpha }}k|<C_1\langle k\rangle ^{\frac{\alpha }{1-\alpha }}; \\
\text{supp}\eta _{k}^{\alpha }\subset \{\xi :|\xi -\langle k\rangle ^{%
\frac{\alpha }{1-\alpha }}k|<C_2\langle k\rangle ^{\frac{\alpha }{1-\alpha }%
}\}; \\
\sum_{k\in \mathbb{Z}^{d}}\eta _{k}^{\alpha }(\xi )\equiv 1,\forall \xi \in
\mathbb{R}^{d}; \\
|\partial ^{\gamma }\eta _{k}^{\alpha }(\xi )|\leq C_{\gamma}\langle
k\rangle ^{-\frac{\alpha |\gamma |}{1-\alpha }},\forall \xi \in \mathbb{R}%
^{d},\gamma \in \mathbb{N}^{d},%
\end{cases}%
\een
 for suitable $C_2>C_1>0$.
 We refer to \cite{HanWang2014JMSJ} for more details.
	Here, the exact smooth partition is chosen for the convenience of our proof.
\end{remark}

Suppose $p\in (0, \infty]$.
Let $\{a_k\}_{k\in \zd}$ denote a sequence of complex
numbers. Set $\Vert \{a_k\}\Vert _{l^p(\zd)}=\left( \sum_{k\in \zd}|a_k|^{p}\right) ^{\frac{1}{p}}$ with usual modification for $p=\fy$.
We use $l^{p}(\zd)$ to denote the set of all sequences $\{a_{k}\}_{k\in \zd}$ such that $\Vert \{a_k\}\Vert _{l^p(\zd)}<\infty$.

For a bounded subset $\Om$ of $\rd$, we use $\calS'_{\Om}(\rd)$ to denote the collection of all distribution $f\in \calS'(\rd)$ such that $\text{supp}\hat{f}\subset \Om$.
Correspondingly, we use the notation $\calS_{\Om}=\calS\cap \calS'_{\Om}$ and $L^p_{\Om}=L^p\cap \calS'_{\Om}$.
Let $B(\xi_0,R)$ denotes the open ball centered at $\xi_0$ with radius $R$, where $R>0$.
\begin{lemma}\cite[Theorem 1.4.1(3)]{Triebel1983}\label{lm-ebd-Lp}
 Let $0<p\leq
q\leq \infty $ and assume $f\in L^p_{B(\xi_0,R)}(\rd)$. We have
\begin{equation}
\Vert f\Vert _{L^{q}(\rd)}\leq CR^{d(1/p-1/q)}\Vert
f\Vert _{L^p(\rd)},
\end{equation}%
where $C$ is independent of $f$ and $R$.
\end{lemma}


\section{Boundness kernel theorem for $\alpha$-modulation spaces}
In this section we establish the boundness kernel theorem for linear operators between $\al$-modulation spaces.
The basic idea comes from transforming the action of linear operator on the function space into the action on the atoms.
We first establish some expansions with unconditional convergence in function spaces. 
Let $N\in \mathbb{N}$, the function space $\calC^N$ is defined by
\be
\calC^N:=\{g\in C^N: |\partial^{\g}g|\lesssim 1\ \text{for all}\ |\g|\leq N\}.
\ee
with the norm $\|g\|_{\calC^N}:=\sum_{|\g|\leq N}\|\partial^{\g}g\|_{L^{\fy}}$.

\begin{lemma}\label{lm-Fourierexp-standard}
    Let $p\in (0,\fy]$, $r\in (0,1)$.
	Let $\psi$ be a smooth function supported on $B(0,1/2)$, satisfying $\psi\equiv 1$ on $B(0,r/2)$.
	Let $\la\in (0,1]$.
	For any $f\in \calS_{B(0,r/2)}(\rd)$, we have the expansion
	\be
	f(x)=\la^d\sum_{k\in \zd}f(\la k)T_{\la k}\check{\psi}(x)
	\ee
	with unconditional convergence in $\calS(\rd)$. Moreover, for any subset $\G\subset \zd$
    and $N>\frac{d}{p\wedge 1}$,
    we have the estimate
    \ben\label{lm-Fourierexp-standard-cd2}
    \la^d\|\sum_{k\in \G}f(\la k)T_{\la k}\check{\psi}(x)\|_{L^p(\rd)}\lesssim 
    \la^{d/p}\|\psi(\frac{\cdot}{\la})\|_{\calC^N}\|\{f(\la k)\}_{k\in \G}\|_{l^p(\zd)},
    \een
    where the implicit constant is independent of $f$, $\psi$ and $\la$.
 For every fixed $p\in (0,\fy]$, we have the equivalent relation
	\ben\label{lm-Fourierexp-standard-cd1}
	\|f\|_{L^p(\rd)}\sim_r \la^{d/p}\|\{f(\la k)\}_{k\in \zd}\|_{l^p(\zd)},\ \ \ f\in \calS_{B(0,r/2)}(\rd),
	\een
 where the implicit constant is independent of $f$ and $\la$.
\end{lemma}
\begin{proof}
	Using the expansion of Fourier series, we write
	\be
	\hat{f}(\xi/\la)=\la^d\sum_{k\in \zd}f(-\la k)e^{2\pi ik\cdot \xi}\psi(\xi/\la)=\la^d\sum_{k\in \zd}f(\la k)e^{-2\pi ik\cdot \xi}\psi(\xi/\la),
	\ee
	which implies that
	\be
	\hat{f}(\xi)=\la^d\sum_{k\in \zd}f(\la k)e^{-2\pi i\la k\cdot \xi}\psi(\xi),
	\ee
	where the series converges uniformly for all $\xi\in \rd$, and unconditional convergences in $\calS(\rd)$ by the rapidly decreasing of $f$.
	Applying inverse Fourier transform, we have the desired expansion in the sense of $\calS(\rd)$:
	\be
	f(x)=\la^d\sum_{k\in \zd}f(\la k)T_{\la k}\check{\psi}(x).
	\ee
	Next, we turn to the estimate of \eqref{lm-Fourierexp-standard-cd2}. 
    Let $\psi_{\la}(\xi)=\psi(\xi/\la)$. The inequality \eqref{lm-Fourierexp-standard-cd2} is equivalent to
    \be
    \|\sum_{k\in \G}f(\la k)T_{k}\check{\psi_\la}(x)\|_{L^p(\rd)}
    \lesssim 
    \|\psi_{\la}\|_{\calC^N}    \|\{f(\la k)\}_{k\in \G}\|_{l^p(\zd)}.
    \ee
    Note that
    \be
    |\check{\psi_\la}(x)|\lesssim \|\psi_\la\|_{\calC^N}\lan x\ran^{-N}.
    \ee
    We have
    \be
    \begin{split}
        |\sum_{k\in \G}f(\la k)T_{k}\check{\psi_\la}(x)|
        \lesssim
        \|\psi_\la\|_{\calC^N}|\sum_{k\in \G}|f(\la k)|\lan x-k\ran^{-N}|.
    \end{split}
    \ee
    For $N>\frac{d}{p\wedge 1}$, we conclude that
    \be
    \begin{split}
        \|\sum_{k\in \G}f(\la k)T_{k}\check{\psi_\la}(x)\|_{L^p(\rd)}
        \lesssim &
        \|\psi_\la\|_{\calC^N}\|\{\sum_{k\in \G}|f(\la k)|\lan l-k\ran^{-N}\}_{l\in \zd}\|_{l^p(\zd)}
        \\
        \lesssim &
	\|\psi_\la\|_{\calC^N}\|\{f(\la k)\}_{k\in \G}\|_{l^p(\zd)}\|\{\lan k\ran^{-N}\}_k\|_{l^{p\wedge 1}(\zd)}
        \\
        \lesssim &
        \|\psi_\la\|_{\calC^N}\|\{f(\la k)\}_{k\in \G}\|_{l^p(\zd)}.
    \end{split}
    \ee
	where we use Young's inequality $l^p(\zd)\ast l^{p\wedge 1}(\zd)\subset l^p(\zd)$.

 Finally, we come to prove the equivalent relation \eqref{lm-Fourierexp-standard-cd1}. 
 Using the expansion of $f$ with $\la=1$, we write
	\be
	|f(x)|\lesssim \sum_{k\in \zd}|f(k)T_k\check{\psi}(x)|
	\lesssim 
	\|\psi\|_{\calC^N}\sum_{k\in \zd}|f(k)|\cdot \lan x-k\ran^{-N}
	\lesssim
	\|\psi\|_{\calC^N}\sum_{k\in \zd}|f(k)|\cdot \lan l-k\ran^{-N}
	\ee
    for all $x\in Q_l:=l+[-\frac{1}{2}, \frac{1}{2}]^d$.
	Then, for $N>\frac{d}{p\wedge 1}$, we have
	\be
	\|f\|_{L^p}
	\lesssim 
	\|\psi\|_{\calC^N}\big\|\big\{\sum_{k\in \zd}|f(k)|\cdot \lan l-k\ran^{-N}\big\}_{l\in \zd}\big\|_{l^p(\zd)}
	\lesssim 
	\|\psi\|_{\calC^N}\|\{f(k)\}_{k\in \zd}\|_{l^p(\zd)}.
	\ee
	
	On the other hand, write
	\be
	f(k)=\scrF^{-1}(\hat{f}\psi)(k)=\int_{\rd}f(y)\check{\psi}(k-y)dy.
	\ee
	For $p\leq 1$, we use the embedding $L^p_{\Om}\subset L^1_{\Om}$ to obtain that
	\be
	|f(k)|\leq \|f(y)\check{\psi}(k-y)\|_{L^1}\lesssim \|f(y)\check{\psi}(k-y)\|_{L^p}
	\lesssim \|\psi\|_{\calC^N}\|f(y)\lan k-y\ran^{-N}\|_{L^p}.
	\ee
	For $p> 1$, we use the H\"{o}lder inequality to obtain that
	\be
    \begin{split}
        |f(k)|\leq \|f(y)\check{\psi}(k-y)\|_{L^1}
	\lesssim &
	\|\psi\|_{\calC^{N+M}}\|f(y)\lan k-y\ran^{-(N+M)}\|_{L^1}
     \\
	\lesssim &
	\|\psi\|_{\calC^{N+M}}\|f(y)\lan k-y\ran^{-N}\|_{L^p}.
    \end{split}
	\ee
	From the above two estimates, for sufficiently large $N$ and $M$ we find
	\be
	\begin{split}
		\|\{f(k)\}_{k\in \zd}\|_{l^p(\zd)}
		\lesssim
		\|\psi\|_{\calC^{N+M}}\|\{\|f(y)\lan k-y\ran^{-N}\|_{L^p}\}_{k\in \zd}\|_{l^p(\zd)}
		\lesssim \|\psi\|_{\calC^{N+M}}\|f\|_{L^p}.
	\end{split}
	\ee	
 Note that the function $\psi$ can be chosen only depends on $r$. Thus, the 
 implicit constant only depends on $r$.
	Now, we have verified the estimate
	\be
	\|f\|_{L^p(\rd)}\sim \|\{f(k)\}_{k\in \zd}\|_{l^p(\zd)},\ \ \ f\in \calS_{B(0,r/2)}(\rd).
	\ee
	The desired estimate \eqref{lm-Fourierexp-standard-cd1} follows by a dilation argument:
 \be
\|f\|_{\lpd}=\la^{d/p}\|f(\la \cdot)\|_{\lpd}\sim \la^{d/p}\|\{f(\la k)\}_{k\in \zd}\|_{l^p(\zd)},
 \ee
 where we use the fact that $f(\la \cdot)\subset \calS_{B(0,r/2)}(\rd)$ for all $\la\in (0,1]$.
\end{proof}

\begin{proposition}\label{pp-Fourierexp-standard}
	Let $p\in (0,\fy]$, $r\in (0,1)$.
	Let $\psi$ be a smooth function supported on $B(0,1/2)$, satisfying $\psi\equiv 1$ on $B(0,r/2)$.
	Let $\la\in (0,1]$.
	For any $f\in L^p_{B(0,r/2)}(\rd)$, we have the expansion
	\ben\label{pp-Fourierexp-standard-cd0}
	f(x)=\la^d\sum_{k\in \zd}f(\la k)T_{\la k}\check{\psi}(x)
	\een
	with unconditional convergence in $L^p(\rd)$ for $p<\fy$, and weak-star convergence in $L^{\fy}(\rd)$. 
 Moreover, for any subset $\G\subset \zd$
    and $N>\frac{d}{p\wedge 1}$,
    we have the estimate
    \ben\label{pp-Fourierexp-standard-cd2}
    \la^d\|\sum_{k\in \G}f(\la k)T_{\la k}\check{\psi}(x)\|_{L^p(\rd)}\lesssim 
    \la^{d/p}\|\psi(\frac{\cdot}{\la})\|_{\calC^N}\|\{f(\la k)\}_{k\in \G}\|_{l^p(\zd)},
    \een
    where the implicit constant is independent of $f$, $\psi$ and $\la$.
	For every fixed $p\in (0,\fy]$, we have the equivalent relation
	\ben\label{pp-Fourierexp-standard-cd1}
	\|f\|_{L^p(\rd)}\sim_r  \la^{d/p}\|\{f(\la k)\}_{k\in \zd}\|_{l^p(\zd)},\ \ \ f\in L^p_{B(0,\frac{r}{2})}(\rd),
	\een
  where the implicit constant is independent of $f$ and $\la$.
	Conversely, if $g\in \calS'_{B(0,\frac{r}{2})}(\rd)$ satisfies $\{g(\la k)\}_{k\in \zd}\in l^p(\zd)$ for some $\la \in (0,1]$, we have $g \in L^p_{B(0,\frac{r}{2})}(\rd)$.
\end{proposition}
\begin{proof}
    For $f\in L^p_{B(0,r/2)}(\rd)$, we have $f\in L^p_{B(0,r/2-\ep)}(\rd)$ for sufficiently small $\ep>0$.
	Choose a function $\r\in \calS(\rd)$ such that $\r(0)=1$ and $\text{supp}\hat{\r}\subset B(0,1)$. Let $\r_N(x)=\r(x/N)$.
	Denote $f_N=\r_Nf$. We have $f_N\in \calS(\rd)$ and
	$\text{supp}\hat{f_N}\subset B(0,r/2-\ep)+B(0,1/N)\subset B(0,r/2)$ for sufficiently large $N$. Using Lemma \ref{lm-Fourierexp-standard} for $f_N$, we obtain the estimate
	\be
	\|f_N\|_{L^p(\rd)}\sim_r  \la^{d/p}\|\{f_N(\la k)\}_{k\in \zd}\|_{l^p(\zd)}.
	\ee
	The desired estimate \eqref{pp-Fourierexp-standard-cd1} follows by
	\be
	\begin{split}
	\la^{d/p}\|\{f(\la k)\}_{k\in \zd}\|_{l^p(\zd)}
	\lesssim & 
	\liminf_{N\rightarrow \fy}\la^{d/p}\|\{f_N(\la k)\}_{k\in \zd}\|_{l^p(\zd)}
	\\
	\sim_r &
	\liminf_{N\rightarrow \fy}\|f_N\|_{L^p(\rd)}
	 \lesssim \|f\|_{L^p(\rd)},
	\end{split}
	\ee
	and that
		\be
	\begin{split}
     \|f\|_{L^p(\rd)}
     \leq\liminf_{N\rightarrow \fy}\|f_N\|_{L^p(\rd)}
     \sim \liminf_{N\rightarrow \fy}\la^{d/p}\|\{f_N(\la k)\}_{k\in \zd}\|_{l^p(\zd)}
     \lesssim \la^{d/p}\|\{f(\la k)\}_{k\in \zd}\|_{l^p(\zd)},
	\end{split}
	\ee
    where the implicit constants in the above three estimates are all independent of $f$ and $\la$.
	Moreover, using \eqref{lm-Fourierexp-standard-cd2} in Lemma \ref{lm-Fourierexp-standard}, 
	for any finite subset $E\subset \zd$
	we obtain that for $p\in(0, \i]$
	\be
	\begin{split}
	\|f_N(x)-\la^d\sum_{k\in E}f_N(\la k)T_{\la k}\check{\psi}(x)\|_{L^p(\rd)}
	= &
	\|\la^d\sum_{k\in E^c}f_N(\la k)T_{\la k}\check{\psi}(x)\|_{L^p(\rd)}
	\\
	\lesssim &
	\|\{f_N(\la k)\}_{k\in E^c}\|_{l^p(\zd)}\leq \|\r\|_{L^{\fy}}\|\{f(\la k)\}_{k\in E^c}\|_{l^p(\zd)}.
	\end{split}
	\ee
	Letting $N\rightarrow \fy$ and applying the Lebesgue dominated convergence theorem, we conclude that for $p\in(0, \i)$
		\be
	\begin{split}
		\|f(x)-\la^d\sum_{k\in E}f(\la k)T_{\la k}\check{\psi}(x)\|_{L^p(\rd)}
		\lesssim 
		\|\r\|_{L^{\fy}}\|\{f(\la k)\}_{k\in E^c}\|_{l^p(\zd)}.
	\end{split}
	\ee
	Hence, \eqref{pp-Fourierexp-standard-cd0} is valid with unconditional convergence in $L^p(\rd)$ for $p<\fy$.
	The case $p=\fy$ can be verified by a similar argument with suitable modification.

    Note that the inequality \eqref{pp-Fourierexp-standard-cd2} can be deduced by the same method as in \eqref{lm-Fourierexp-standard-cd2}.
    
	Finally, we turn to the inverse direction. For $g\in \calS'_{B(0,\frac{r}{2})}$, we write $g_N=\r_Ng$.
	Similarly, we have $g_N\in \calS(\rd)$ and $\text{supp}\hat{g_N}\subset B(0,\frac{r}{2})$ for sufficiently large $N$.
	We also have
	\be
	\|g_N\|_{L^p(\rd)}\sim \la^{d/p}\|\{g_N(\la k)\}_{k\in \zd}\|_{l^p(\zd)}\leq \la^{d/p}\|\r\|_{L^{\fy}}\|\{g(\la k)\}_{k\in \zd}\|_{l^p(\zd)}.
	\ee
	The desired conclusion follows by letting $N\rightarrow \fy$ and using the Fatou lemma
	\be
	\|g\|_{L^p(\rd)} \lesssim \liminf_{N\rightarrow \fy}\|g_N\|_{L^p(\rd)}\lesssim \la^{d/p}\|\{g(\la k)\}_{k\in \zd}\|_{l^p(\zd)}.
	\ee
\end{proof}

\begin{proposition}\label{pp-Fourierexp-alpha}
	Let $p\in (0,\fy]$, $r>1$. Let $\psi$ be a smooth function supported on $B(0,1/2)$, satisfying $\psi\equiv 1$ on $B(0,\frac{1}{2r})$.
	Let $\la\in (0,1]$.
	For any $f\in L^p_{B(\xi_0,R)}(\rd)$, we have the expansion
	\ben\label{pp-Fourierexp-alpha-cd0}
	f(x)
	=
	\Big(\frac{\la}{2rR}\Big)^{d}
	 \sum_{k\in \zd}
	  f\Big(\frac{\la k}{2rR}\Big) T_{\frac{\la k}{2rR}} \scrF^{-1}(\psi_{\xi_0,2rR})(x),\ \ \ \  
	 \psi_{\xi_0,2rR}(\xi)=\psi\Big(\frac{\xi-\xi_0}{2rR}\Big),
	\een
	with unconditional convergence in $L^p(\rd)$ for $p<\fy$, and weak-star convergence in $L^{\fy}(\rd)$. 
	Moreover, for any subset $\G\subset \zd$
    and $N>\frac{d}{p\wedge 1}$,
    we have the estimate
    \ben\label{pp-Fourierexp-alpha-cd2}
    \Big(\frac{\la}{rR}\Big)^{d}\|\sum_{k\in \G}
	  f\Big(\frac{\la k}{2rR}\Big) T_{\frac{\la k}{2rR}} \scrF^{-1}(\psi_{\xi_0,2rR})(x)\|_{L^p(\rd)}\lesssim 
    \|\psi(\frac{\cdot}{\la})\|_{\calC^N}\Big(\frac{\la}{rR}\Big)^{d/p} 
	  \Big\|\Big\{f\Big(\frac{\la k}{2rR}\Big)\Big\}_{k\in \G} \Big\|_{l^p(\zd)},
    \een
    where the implicit constant is independent of $f$, $\psi$ and $\la$.
    For every fixed $p\in (0,\fy]$, we have the equivalent relation
	\ben\label{pp-Fourierexp-alpha-cd1}
	\|f\|_{L^p(\rd)}
	\sim_r
	  \Big(\frac{\la}{rR}\Big)^{d/p} 
	  \Big\|\Big\{f\Big(\frac{\la k}{2rR}\Big)\Big\}_{k\in \zd} \Big\|_{l^p(\zd)},
	\ \ \ f\in L^p_{B(\xi_0,R)}(\rd),
	\een
    where the implicit constant is independent of $f$, $\la$, $\xi_0$, $r$ and $R$.
    If $f\in \calS_{B(\xi_0,R)}(\rd)$ the expansion \eqref{pp-Fourierexp-alpha-cd0} unconditional converges in the topology of $\calS(\rd)$.
	
	Conversely, if $g\in \calS'_{B(\xi_0,R)}(\rd)$ satisfies $\{g(\frac{\la k}{2rR})\}_{k\in \zd}\in l^p(\zd)$ for some $\la \in (0,1]$, we have $g \in L^p_{B(\xi_0,R)}(\rd)$.
\end{proposition}
\begin{proof}
	Let $\hat{h}(\xi)=\hat{f}(2rR\xi+\xi_0)$.
	We have $h\in L^p_{B(0,\frac{1}{2r})}$. Using Proposition \ref{pp-Fourierexp-standard}, we have the expansion
	\be
    h(x)=\la^d\sum_{k\in \zd}h(\la k)T_{\la k}\check{\psi}(x).
    \ee
    By a direct calculation, we obtain
    \be
    h(x)=(2rR)^{-d}e^{\frac{-2\pi ix\cdot \xi_0}{2rR}}f(\frac{x}{2rR})
    \ee
    and 
    \be
    e^{\frac{-2\pi ix\cdot \xi_0}{2rR}}f(\frac{x}{2rR})=\la^d\sum_{k\in \zd}e^{\frac{-2\pi i\la k\cdot \xi_0}{2rR}}f(\frac{\la k}{2rR})T_{\la k}\check{\psi}(x).
    \ee
    Replacing $x$ by $2rRx$, we conclude 
    \be
    e^{-2\pi ix\cdot \xi_0}f(x)=\la^d\sum_{k\in \zd}e^{\frac{-2\pi i\la k\cdot \xi_0}{2rR}}f(\frac{\la k}{2rR})\check{\psi}(2rRx-\la k),
    \ee
    which implies the desired expansion by
    \be
    \begin{split}
    	f(x)
    	= &
    	\la^d\sum_{k\in \zd}e^{\frac{-2\pi i\la k\cdot \xi_0}{2rR}}f(\frac{\la k}{2rR})e^{2\pi ix\cdot \xi_0}\check{\psi}(2rRx-\la k)
    	\\
    	= &
    	\Big(\frac{\la}{2rR}\Big)^{d}
    	\sum_{k\in \zd}f(\frac{\la k}{2rR})T_{\frac{\la k}{2rR}}\big( (2rR)^de^{2\pi ix\cdot \xi_0}\check{\psi}(2rRx)\big)
    	\\
    	= &
    	\Big(\frac{\la}{2rR}\Big)^{d}
    	\sum_{k\in \zd}f(\frac{\la k}{2rR})T_{\frac{\la k}{2rR}}\scrF^{-1}(\psi_{\xi_0,2rR})(x).
    \end{split}
    \ee
    On the other hand, the desired norm estimate \eqref{pp-Fourierexp-alpha-cd1} follows by 
    \be
    \|h\|_{L^p(\rd)}\sim_r  \la^{d/p}\|\{h(\la k)\}_{k\in \zd}\|_{l^p(\zd)},
    \ee
    \be
    \|h\|_{L^p(\rd)}=(2rR)^{-d}\|f(\frac{x}{2rR})\|_{L^p(\rd)}\sim (rR)^{d(1/p-1)}\|f\|_{L^p(\rd)}
    \ee
    and
    \be
    \|\{h(\la k)\}_{k\in\zd}\|_{l^p(\zd)}
    \sim (rR)^{-d}\Big\|\Big\{f(\frac{\la k}{2rR})\Big\}_{k\in\zd}\Big\|_{l^p(\zd)}.
    \ee
    Using a dilation argument,  the inequality \eqref{pp-Fourierexp-alpha-cd2} is equivalent to
    \be
    \|\sum_{k\in \G}
	  f\Big(\frac{\la k}{2rR}\Big) T_{k} \scrF^{-1}(\psi(\frac{\cdot}{\la}-\frac{\xi_0}{2rR})))(x)\|_{L^p(\rd)}\lesssim 
    \|\psi(\frac{\cdot}{\la})\|_{\calC^N}
	  \Big\|\Big\{f\Big(\frac{\la k}{2rR}\Big)\Big\}_{k\in \zd} \Big\|_{l^p(\zd)},
    \ee 
    which can be deduced by the same method as in \eqref{pp-Fourierexp-standard-cd2}.

    The remaining results follow by a similar dilation argument.
\end{proof}

Now, we are in a position to give the boundedness kernel theorem.
Suppose that $r>1$.
Let $\{\va_{k}^{\al}\}_{k\in \zd}$ be a Schwartz function sequence satisfying that
\begin{equation}\label{dual-frame-alpha}
\begin{cases}
\va_k^{\al}(\xi)=1\ on\  B(\lan k\ran^{\frac{\al}{1-\al}}k, C\lan k\ran^{\frac{\al}{1-\al}}); \\
\text{supp}\va_k^{\al}\subset B(\lan k\ran^{\frac{\al}{1-\al}}k, rC\lan k\ran^{\frac{\al}{1-\al}}); \\
|\partial ^{\gamma }\va _{k}^{\alpha }(\xi )|\leq C_{\gamma}\langle
k\rangle ^{-\frac{\alpha |\gamma |}{1-\alpha }},\forall \xi \in \mathbb{R}%
^{d},\gamma \in \mathbb{N}^{d},
\end{cases}%
\end{equation}%
where $C$ is the constant in \eqref{pre-1}. 
In the subsequent part of this article, we will use the notation $\va_k^{\al}$ and $r$ with the meaning mentioned above. 

\begin{theorem}[Boundedness kernel theorem for $\alpha$-modulation spaces]\label{thm-BKa}
	Suppose that $p,q\in (0,\fy]$ satisfying $p\leq q\wedge 1$.
	Let $A\in \calL(\calS(\rd), \calS'(\rd))$ be a linear operator with
	distributional kernel $K_A\in \calS'(\rdd)$.	
	For any $\la\in (0, 1]$, denote $\b_k=\frac{\la}{2rC\lan k\ran^{\frac{\al}{1-\al}}}$ for all $k\in \zd$.
	The following statements are equivalent.
	\bn
	\item
	$A\in\calL(M^p_{\al}(\rd), M^q_{\al}(\rd))$;
	\item
	$\{\lan k\ran^{\frac{\al d}{1-\al}(1/p-1)}A(T_{\b_k k'} \check{\va_k^{\al}}): k,k' \in \zd\}$
	is a bounded subset of $M^{q}_{\al}(\rd)$;
	\item
	$\sup\limits_{k,k' \in \zd}\lan k\ran^{\frac{\al d}{1-\al}(1/p-1)}
		\big\|\big\{\lan l\ran^{\frac{-\al d}{(1-\al)q}}\lan K_A(z',z),  \overline{\check{\eta_l^{\al}}(\b_l l'-z')\otimes\check{\va_k^{\al}}(z-\b_k k')}  \ran\big\}_{l, l'\in \zd}\big\|_{l^q(\zdd)}<\fy$.
	\item 
	$\big\{
	  \big\{
	    \lan k\ran^{\frac{\al d}{1-\al}(1/p-1)}
	    \lan l\ran^{\frac{-\al d}{(1-\al)q}}\lan K_A(z',z),  \overline{\check{\eta_l^{\al}}(\b_l l'-z')\otimes\check{\va_k^{\al}}(z-\b_k k')}  \ran
	  \big\}_{l, l'\in \zd}: 
	  k,k' \in \zd
	 \big\}$
	is a bounded subset of $l^q(\zdd)$.
	\en
    Moreover, if one of the above statements holds, we have the norm estimate
    \ben\label{thm-BKa-norm}
    \begin{split}
    &\|A\|_{\calL(\mpad,\mqad)}
    \sim 
    \sup_{k,k'\in\zd}\lan k\ran^{\frac{\al d}{1-\al}(1/p-1)}
  \|A(T_{\b_k k'}\check{\va_k^{\al}})\|_{M^q_{\al}(\rd)}
  \\
  \sim &
  \sup\limits_{k,k' \in \zd}\lan k\ran^{\frac{\al d}{1-\al}(1/p-1)}
		\big\|\big\{\b_l^{d/q}\lan K_A(z',z),  \overline{\check{\eta_l^{\al}}(\b_l l'-z')\otimes\check{\va_k^{\al}}(z-\b_k k')}  \ran\big\}_{l, l'\in \zd}\big\|_{l^q(\zdd)}.
    \end{split}
    \een
    where the implicit constant is independent of $\la$.
	\end{theorem}

\begin{proof}
	First, the relation $(1)\Longrightarrow (2)$ follows by $A\in \calL(\mpad,\mqad)$
	and the fact that
	\be
	\{\lan k\ran^{\frac{\al d}{1-\al}(1/p-1)}T_y\check{\va_k^{\al}}\}_{k\in \zd, y\in \rd}
	\ee
	is a bounded subset of $\mpad$.
    More precisely, we have the estimate
    \ben\label{thm-BKa-pf1}
    \begin{split}
        \lan k\ran^{\frac{\al d}{1-\al}(1/p-1)} \|A(T_{\b_k k'}\check{\va_k^{\al}})\|_{\mqad}
        \lesssim &
        \|A\|_{\calL(\mpad,\mqad)}\lan k\ran^{\frac{\al d}{1-\al}(1/p-1)} \|T_{\b_k k'}\check{\va_k^{\al}}\|_{\mpad}
        \\
        \lesssim &
        \|A\|_{\calL(\mpad,\mqad)}\lan k\ran^{\frac{\al d}{1-\al}(1/p-1)} \|\check{\va_k^{\al}}\|_{L^p}
        \\
        \lesssim &
        \|A\|_{\calL(\mpad,\mqad)}.
    \end{split}
    \een
	
	Next, we turn to the relation $(2)\Longrightarrow (1)$.
	For a fixed function $f\in \calS(\rd)$, we have $\text{supp}\scrF(\Box_k^{\al}f)\subset B(\lan k\ran^{\frac{\al}{1-\al}}k, C\lan k\ran^{\frac{\al}{1-\al}})$.
	Write $\b_k=\frac{\la}{2rC\lan k\ran^{\frac{\al}{1-\al}}}$.
	Using Proposition \ref{pp-Fourierexp-alpha} with 
   $\xi_0=\lan k\ran^{\frac{\al}{1-\al}}k$, $R=C\lan k\ran^{\frac{\al}{1-\al}}$, 
   $\psi(\xi)=\va_k^{\al}(2rC\lan k\ran^{\frac{\al}{1-\al}}\xi+\lan k\ran^{\frac{\al}{1-\al}}k)$
   and
   $\la\in(0,1]$, we have the expansion in $\calS(\rd)$:
	\ben\label{thm-Kernel-alpha-BD-pf-1}
	\Box_k^{\al}f(x)=\b_k^d\sum_{k'\in \zd}\Box_k^{\al}f(\b_k k')T_{\b_k k'}\check{\va_k^{\al}}(x),
	\een
	and the norm estimate
	\be
	\|\Box_k^{\al}f\|_{L^p(\rd)}
	\sim_r
	\b_k^{d/p}\|\{\Box_k^{\al}f(\b_k k')\}_{k'\in \zd}\|_{l^p(\zd)}.
	\ee
	Recall that $A\in \calL(\calS(\rd),\calS'(\rd))$.
	Applying the operator $A$ on both sides of the expansion \eqref{thm-Kernel-alpha-BD-pf-1}, we obtain
	\be
	A(\Box_k^{\al}f)=\b_k^d\sum_{k'\in \zd}\Box_k^{\al}f(\b_k k')A(T_{\b_k k'}\check{\va_k^{\al}}),
	\ee
	where the series converges in the topology of $\calS'(\rd)$.
	Taking the $M^q_{\al}(\rd)$ norm and using the (quasi-)triangle inequality, we further conclude that
	\be
	\begin{split}
		\|A(\Box_k^{\al}f)\|_{M^q_{\al}(\rd)}
		\lesssim &
        \b_k^d
		\big\|\{\Box_k^{\al}f(\b_k k')
		\|A(T_{\b_k k'}\check{\va_k^{\al}})\|_{M^q_{\al}(\rd)}\}_{k'}\big\|_{l^{q\wedge 1}}
		\\
		\lesssim &
        \sup_{k'}\|A(T_{\b_k k'}\check{\va_k^{\al}})\|_{M^q_{\al}(\rd)}
		\b_k^{d(1-\frac{1}{q\wedge 1})}
        \b_k^{\frac{d}{q\wedge 1}}
		\|\{\Box_k^{\al}f(\b_k k')
		\}_{k'}\|_{l^{q\wedge 1}}
		\\
        \sim_r &
         \sup_{k'}\|A(T_{\b_k k'}\check{\va_k^{\al}})\|_{M^q_{\al}(\rd)}
		\b_k^{d(1-\frac{1}{q\wedge 1})}
        \|\Box_k^{\al}f\|_{L^{q\wedge 1}}
        \\
		\lesssim &
         \sup_{k'}\|A(T_{\b_k k'}\check{\va_k^{\al}})\|_{M^q_{\al}(\rd)}
		\b_k^{d(1-\frac{1}{q\wedge 1})}\lan k\ran^{\frac{\al d}{1-\al}(\frac{1}{p}-\frac{1}{q\wedge 1})}
        \|\Box_k^{\al}f\|_{L^{p}}
		\\
		\sim &
		\la^{d(1-\frac{1}{q\wedge 1})}\lan k\ran^{\frac{\al d}{1-\al}(1/p-1)}
  \sup_{k'}\|A(T_{\b_k k'}\check{\va_k^{\al}})\|_{M^q_{\al}(\rd)}
        \|\Box_k^{\al}f\|_{L^{p}}.
	\end{split}
	\ee
    For small $\la$, there exits a constant $N\in \bbN$ such that $2^N\la\in [1/2,1]$.
    By repeating the above estimates associated with $2^N\la$ instead of $\la$, we obtain an improved estimate:
    \be
    \begin{split}
		\|A(\Box_k^{\al}f)\|_{M^q_{\al}(\rd)}
		\lesssim &
        (2^N\la)^{d(1-\frac{1}{q\wedge 1})}\lan k\ran^{\frac{\al d}{1-\al}(1/p-1)}
  \sup_{k'}\|A(T_{\b_k 2^Nk'}\check{\va_k^{\al}})\|_{M^q_{\al}(\rd)}
        \|\Box_k^{\al}f\|_{L^{p}}
        \\
        		\lesssim &
        \lan k\ran^{\frac{\al d}{1-\al}(1/p-1)}
  \sup_{k'}\|A(T_{\b_k k'}\check{\va_k^{\al}})\|_{M^q_{\al}(\rd)}
        \|\Box_k^{\al}f\|_{L^{p}}.
    \end{split}
    \ee
	Recalling that $f\in \calS(\rd)$ and $A\in \calL(\calS(\rd),\calS'(\rd))$, we obtain that
	\be
	f=\sum_{k\in \zd}\Box_k^{\al}f, \text{ and } 
	Af=\sum_{k\in \zd}A(\Box_k^{\al}f),
	\ee
	which converge unconditionally in $\calS(\rd)$ and $\calS'(\rd)$, respectively.
	The desired conclusion follows by
	\ben\label{thm-BKa-pf2}
	\begin{split}
		\|Af\|_{M^q_{\al}(\rd)}
		= &
		\big\|\sum_{k\in \zd}A(\Box_k^{\al}f)\big\|_{M^q_{\al}(\rd)}
		\lesssim 
		\big\|\big\{\|A(\Box_k^{\al}f)\|_{M^q_{\al}(\rd)}\big\}_k\big\|_{l^{q\wedge 1}}
  \\
		\lesssim &
		\sup_{k',k}\big(\lan k\ran^{\frac{\al d}{1-\al}(1/p-1)}
  \|A(T_{\b_k k'}\check{\va_k^{\al}})\|_{M^q_{\al}(\rd)}\big)
  \big\|\big\{\|\Box_k^{\al}f\|_{L^p(\rd)}\big\}_k\big\|_{l^{p}}
  \\
  \sim & 
  \sup_{k',k}\big(\lan k\ran^{\frac{\al d}{1-\al}(1/p-1)}
  \|A(T_{\b_k k'}\check{\va_k^{\al}})\|_{M^q_{\al}(\rd)}\big)
  \|f\|_{M^p_{\al}(\rd)}.
	\end{split}
	\een
	
	Next, we turn to the proof of $(2)\Longrightarrow (3)$.
	Recall that
	\be
	\text{supp}(\scrF(\Box_l^{\al}(A(T_y\check{\va_k^{\al}}))))\subset B(\lan l\ran^{\frac{\al}{1-\al}}l, C\lan l\ran^{\frac{\al}{1-\al}}).
	\ee
	Using Proposition \ref{pp-Fourierexp-alpha} for each $\Box_l^{\al}(A(T_y\check{\va_k^{\al}}))$, we obtain
	\ben\label{thm-BKa-pf3}
	\begin{split}
		\|A(T_y\check{\va_k^{\al}})\|_{M^q_{\al}(\rd)}
		= &
		\bigg(\sum_{l\in \zd}\|\Box_l^{\al}(A(T_y\check{\va_k^{\al}}))\|_{L^q(\rd)}^q\bigg)^{1/q}
		\\
		\sim &
		\bigg(\sum_{l,l'\in \zd}|\b_l^{d/q}\Box_l^{\al}(A(T_y\check{\va_k^{\al}}))(\b_l l')|^q\bigg)^{1/q},
	\end{split}
	\een
	where we use the fact that $\b_l=\frac{\la}{2rC\lan l\ran^{\frac{\al}{1-\al}}}$ for some $\la \in (0,1]$.
	From this, the desired conclusion follows by the fact
	\ben\label{thm-BKa-pf4}
	\begin{split}
		\Box_l^{\al}(A(T_y\check{\va_k^{\al}}))(x)
        &
        =\check{\eta_l^{\al}}\ast(A(T_y\check{\va_k^{\al}}))(x)
		=\lan A(T_y\check{\va_k^{\al}}), \overline{\check{\eta_l^{\al}}(x-\cdot)} \ran
		\\ & 
		=\lan K_A(z',z), \overline{\check{\eta_l^{\al}}(x-z')\otimes \check{\va_k^{\al}}(z-y)} \ran,
	\end{split}
    \een
    where $K_A$ is the distributional kernel of $A$. Note that the relation $(3)\Longrightarrow (2)$ also follows by Proposition \ref{pp-Fourierexp-alpha} (the inverse part) with a similar argument. It is obvious that $(3)$ and $(4)$ are equivalent.
    Finally, the norm estimate follows by \eqref{thm-BKa-pf1}, \eqref{thm-BKa-pf2} and \eqref{thm-BKa-pf3}.
    So we complete the proof.
\end{proof}

Next, we prepare some propositions for the duality argument.
\begin{proposition}[Duality of boundedness]\label{pp-dual-bd}
    Let $T, \widetilde{T} \in \calL(\calS(\rd),\calS'(\rd))$ satisfy 
    \be
    \overline{\lan Tf, g\ran}=\lan \widetilde{T}g, f\ran,\ \ \  \forall \ f,g\in \calS(\rd).
    \ee
    Suppose that $1\leq p,q\leq \fy$.  Then the following two statements are equivalent
    \bn
    \item $T\in \calL(\mpad, \mqad)$,
    \\
    \item $\widetilde{T}\in \calL(M^{q'}_{\al}(\rd), M^{p'}_{\al}(\rd))$.
    \en
    Moreover, if one of the above statements holds, 
    we have 
    \be
    \|T\|_{\calL(\mpad, \mqad)}\sim \|\widetilde{T}\|_{\calL(M^{q'}_{\al}(\rd), M^{p'}_{\al}(\rd))}.
    \ee
\end{proposition}
\begin{proof}
    If the statement (1) holds, for any $f,g\in \calS(\rd)$ we have
    \be
    \begin{split}
        |\lan \widetilde{T}g, f\ran|
        = &
        |\lan Tf, g\ran|
        \\
        \lesssim &
        \|Tf\|_{\mqad}\|g\|_{M^{q'}_{\al}(\rd)}
        \\
        \lesssim &
        \|T\|_{\calL(\mpad, \mqad)}\|f\|_{\mpad}\|g\|_{M^{q'}_{\al}(\rd)}.
    \end{split}
    \ee
    Using the duality relation $(\calM^p_{\al}(\rd))^{\ast}=M^{p'}_{\al}(\rd)$ (see Proposition \ref{pp-dual-alph}), we find $\widetilde{T}g\in M^{p'}_{\al}(\rd)$ satisfying 
    \be
    \|\widetilde{T}g\|_{M^{p'}_{\al}(\rd)}\lesssim \|T\|_{\calL(\mpad, \mqad)}\|g\|_{M^{q'}_{\al}(\rd)}, \ \ \ \forall\ g\in \calS(\rd).
    \ee
    From this, we obtain the statement with 
    \be
    \|\widetilde{T}\|_{\calL(M^{q'}_{\al}(\rd), M^{p'}_{\al}(\rd))}\lesssim \|T\|_{\calL(\mpad, \mqad)}.
    \ee
    Using the same method, one can verify $(2)\Longrightarrow (1)$ and obtain the estimate 
    \be
    \|T\|_{\calL(\mpad, \mqad)}\lesssim \|\widetilde{T}\|_{\calL(M^{q'}_{\al}(\rd), M^{p'}_{\al}(\rd))}.
    \ee
\end{proof}

\begin{proposition}[Duality of $\al$-modulation space]\label{pp-dual-alph}
Let $1\leq p\leq \fy$, we have the duality relation
\ben\label{pp-dual-alph-c1}
(\calM^p_{\al}(\rd))^*=M^{p'}_{\al}(\rd).
\een
More precisely, there is a bounded map $\rho: M^{p'}_{\al}(\rd)\longrightarrow (\calM^{p}_{\al}(\rd))^*$, defined by
  $\rho: h\in M^{p'}_{\al}(\rd) \mapsto \rho(h)$ with
  \be
  (\rho(h))(f): =\overline{\lan h, f\ran} ,\ \ f\in \calS(\rd).
  \ee
  Conversely, there exists another bounded map $\widetilde{\rho}: (\calM^p_{\al}(\rd))^*\longrightarrow M^{p'}_{\al}(\rd)$
  where the image of $T\in (\calM^p_{\al}(\rd))^*$ under $\widetilde{\rho}$ is defined as the (unique) element $h\in M^{p'}_{\al}(\rd)$ such that
  \be
  T(f)=\overline{\lan h, f\ran},\ \ \ \ f\in \calS(\rd).
  \ee
  Moreover, we have $\rho\circ \widetilde{\rho}=I_{(\calM^{p}_{\al}(\rd))^*}$,\ \ $\widetilde{\rho}\circ \rho=I_{M^{p'}_{\al}(\rd)}$.
  So the equality \eqref{pp-dual-alph-c1} is valid in the sense of isomorphism.
\end{proposition}

\begin{proof}
    The case $p\in [1,\fy)$ has been proved in \cite[Theorem 2.1]{HanWang2014JMSJ}. We only need to deal with the case $p=\fy$.
    In this case, $p'=1$, the duality relation should be read as $(\calM^{\fy}_{\al})^*=M^1_{\al}$. 
    For any $h\in M^{1}_{\al}$ and $f\in \calS$, we have 
    \be
     \begin{split}
         |(\rho(h))(f)|
         = &
         |\lan h, f\ran|
         =
         |\sum_{k}\lan h, \Box^{\al}_kf\ran|
         \\
         = &
         |\sum_{k}\sum_{|l|\leq N}\lan \Box^{\al}_{k+l}h, \Box^{\al}_kf\ran|
         \\
         \lesssim &
         \sum_{|l|\leq N}\sum_{k}\|\Box^{\al}_{k+l}h\|_{L^1}\|\Box^{\al}_kf\|_{L^{\fy}}\lesssim \|h\|_{M^{1}_{\al}}\|f\|_{M^{\fy}_{\al}}.
     \end{split}
    \ee
    Note that, the mapping $\rho(h)$ can be extended from Schwartz function to $\calM^{\fy}_{\al}$,  maintaining the following estimate
    \be
    |(\rho(h))(f)|\lesssim \|h\|_{M^{1}_{\al}}\|f\|_{M^{\fy}_{\al}},\ \ \ \ \forall f\in \calM^{\fy}_{\al}.
    \ee
    From this and the fact that $\rho(h)$ is a linear functional on $\calM^{\fy}_{\al}$, we find that $\rho(h)\in (\calM^{\fy}_{\al})^*$ with $\|\r(h)\|_{(\calM^{\fy}_{\al})^*}\lesssim \|h\|_{M^{1}_{\al}(\rd)}$.

    One the other hand, for $T\in (\calM^{\fy}_{\al})^*$, we have
  \be
  |T(f)|\leq \|T\|_{(\calM^{\fy}_{\al})^*}\|f\|_{M^{\fy}_{\al}},\ \ \forall f\in \calS.
  \ee
  Recall that $\calS$ is continuously embedded into $M^{\fy}_{\al}$.
  Thus, the map $T$ restricted on $\calS$ induces a bounded linear functional on $\calS$.
  There exits a tempered distribution, denoted by $h=\widetilde{\r}(T)$, such that
  \be
  T(f)=
  \overline{\langle h, f\rangle}
  ,\ \ \forall f\in \calS(\rd).
  \ee
  Next, we verify that $h\in M^{1}_{\al}$ with $\|h\|_{M^{1}_{\al}}\lesssim \|T\|_{(\calM^{\fy}_{\al})^*}$.
  Let $\G$ be a finite subset of $\zd$, 
   $\{f_k\}_{k\in \G}$ be a sequence of Schwartz function.
  Note that
  \be
   f:=\sum_{k\in \G}\Box_k^{\al}f_k\in \calS,
  \ee
  where the decomposition function $\rho_0$ here is real-valued. We have the estimate
  \be
   |\langle h, f\rangle|=|T(f)|\lesssim \|T\|_{(\calM^{\fy}_{\al})^*}\|f\|_{M^{\fy}_{\al}}
   \lesssim 
   \|T\|_{(\calM^{\fy}_{\al})^*}\sup_{k\in \G}\|f_k\|_{L^{\fy}}.
  \ee
  Write
  \be
  \langle h, f\rangle=\sum_{k\in \G}\lan h, \Box_k^{\al}f_k\ran=\sum_{k\in \G}\lan\Box_k^{\al} h, f_k\ran.
  \ee
  By multiplying suitable constants $e^{i\th_k}$ to $f_k$, we conclude that
  \ben\label{pp-dual-alph-1}
  \begin{split}
   &\sum_{k\in \G}|\lan\Box_k^{\al}h, f_k\ran|
  = 
  |\sum_{k\in \G}\lan\Box_k^{\al}h, e^{i\th_k}f_k\ran|
  \\
  = &
  |\lan h, \sum_{k\in \G}e^{i\th_k}\Box_k^{\al}f_k\ran|
  \lesssim
  \|T\|_{(\calM^{\fy}_{\al})^*}\|\sum_{k\in \G}e^{i\th_k}\Box_k^{\al}f_k\|_{M^{\fy}_{\al}}
   \lesssim 
   \|T\|_{(\calM^{\fy}_{\al})^*}\sup_{k\in \G}\|f_k\|_{L^{\fy}}.
  \end{split}
  \een
  Let $\psi_{\ep}(x)=\ep^{-d}\psi(x/\ep)$ for $\ep\in (0,1)$, where
  $\psi$ is a $C_c^{\fy}$ function satisfying $\int_{\rd}\psi(x)dx=1$.
  Let
  \be
  f_{k,\ep,N}=\big(\text{sgn}(\Box_k^{\al}h\ast \psi_{\ep})\chi_{B(0,N)}\big)\ast \overline{\widetilde{\psi_{\ep}}},
  \ee
  where $\widetilde{\psi_{\ep}}(x)=\psi_{\ep}(-x)$, and
  \be
  \text{sgn}(\Box_k^{\al}h\ast \psi_{\ep})(x)=
  \begin{cases}
      \frac{\Box_k^{\al}h\ast \psi_{\ep}(x)}{|\Box_k^{\al}h\ast \psi_{\ep}(x)|},\ \ &\text{if}\ \Box_k^{\al}h\ast \psi_{\ep}(x)\neq 0,
      \\
      0,  &\text{if}\ \Box_k^{\al}h\ast \psi_{\ep}(x)= 0.
  \end{cases}
  \ee
  Note that $f_{k,\ep,N}\in \calS$ and
  \be
   \lan\Box_k^{\al} h, f_{k,\ep,N}\ran=\lan \Box_k^{\al}h\ast \psi_{\ep}, \text{sgn}(\Box_k^{\al}h\ast \psi_{\ep})\chi_{B(0,N)} \ran 
   =
   \int_{B(0,N)}|\Box_k^{\al}h\ast \psi_{\ep}(x)|dx.
  \ee
   Putting this estimate into \eqref{pp-dual-alph-1}, we obtain that
   \be
   \begin{split}
       \sum_{k\in \G}\int_{B(0,N)}|\Box_k^{\al}h\ast \psi_{\ep}(x)|dx
       = &
       \sum_{k\in \G}|\lan\Box_k^{\al}h, f_{k,\ep,N}\ran|
       \\
       \lesssim &
   \|T\|_{(\calM^{\fy}_{\al})^*}\sup_{k\in \G}\|f_{k,\ep,N}\|_{L^{\fy}}\lesssim \|T\|_{(\calM^{\fy}_{\al})^*},
   \end{split}
   \ee
   where we use the fact that $\|f_{k,\ep,N}\|_{L^{\fy}}\lesssim 1$ uniformly for all $k$, $\ep$, and $N$.
   Letting $\ep\rightarrow 0^+$ and using the Fatou lemma, we find that
   \be
   \begin{split}
   \sum_{k\in \G}\int_{B(0,N)}|\Box_k^{\al}h(x)|dx
   = &
       \sum_{k\in \G}\int_{B(0,N)}\varliminf\limits_{\ep\rightarrow 0^+}|\Box_k^{\al}h\ast \psi_{\ep}(x)|dx
       \\
   \leq &
   \sum_{k\in \G}\varliminf\limits_{\ep\rightarrow 0^+}\int_{B(0,N)}|\Box_k^{\al}h\ast \psi_{\ep}(x)|dx
   \lesssim \|T\|_{(\calM^{\fy}_{\al})^*}.
   \end{split}
   \ee
   By the arbitrary of $\G$ and $N$, we further obtain
   \be
    \sum_{k\in \zd}\int_{\rd}|\Box_k^{\al}h(x)|dx\lesssim \|T\|_{(\calM^{\fy}_{\al})^*},
   \ee
   which imiplies the desired conclusion
   \be
   \|\widetilde{\r}(T)\|_{M^1_{\al}}
   =
   \|h\|_{M^1_{\al}}\lesssim \|T\|_{(\calM^{\fy}_{\al})^*}.
   \ee
   The final step is to verify $\rho\circ \widetilde{\rho}=I_{(\calM^{\fy}_{\al}(\rd))^*}$ and $\widetilde{\rho}\circ \rho=I_{M^1_{\al}(\rd)}$.
   For any $T\in (\calM^{\fy}_{\al})^*$, we have
   \be
   \begin{split}
       (\rho\circ \widetilde{\rho})(T)(\va)
       =
       \overline{\lan \widetilde{\rho}T, \va\ran}=T(\va),\ \ \ \forall\ \va\in \calS.
   \end{split}
   \ee
   From this and the density of $\calS$ in $\calM^{\fy}_{\al}$, we obtain $\rho\circ \widetilde{\rho}=I_{(\calM^{\fy}_{\al}(\rd))^*}$.

   For any $h\in M^1_{\al}$, we have
   \be
   \begin{split}
       \lan(\widetilde{\rho}\circ \rho)(h), \va\ran
       =
       \overline{\rho(h)(\va)}=\lan h,\va\ran,\ \ \ \forall\ \va\in \calS.
   \end{split}
   \ee
   From this, we obtain $\widetilde{\rho}\circ \rho=I_{M^1_{\al}(\rd)}$
   Now, we have completed the proof of this proposition.
\end{proof}

By a duality argument, we have the following conclusion for $A\in\calL(M^p_{\al}(\rd), M^{\i}_{\al}(\rd))$.
\begin{theorem}\label{thm-BKa-dual}
	Suppose that $p\in [1, \i]$.
	Let $A\in \calL(\calS(\rd), \calS'(\rd))$ be a linear operator with
	distributional kernel $K_A\in \calS'(\rdd)$.	
	For any $\la\in (0, 1]$, denote $\b_k=\frac{\la}{2rC\lan k\ran^{\frac{\al}{1-\al}}}$ for all $k\in \zd$.
	The following statements are equivalent.
	\bn
	\item
	$A\in\calL(M^p_{\al}(\rd), M^{\i}_{\al}(\rd))$;
	\item
	$\sup\limits_{k,k' \in \zd}
	\big\|\big\{\lan l\ran^{\frac{-\al d}{(1-\al)p'}}\lan K_A(z',z),  
	\check{\va_k^{\al}}(z'-\b_k k')\otimes\check{\eta_l^{\al}}(\b_l l'-z)  \ran\big\}_{l, l'\in \zd}\big\|_{l^{p'}(\zdd)}<\fy$.
	\en
 Moreover, if one of the above statements holds, we have the norm estimate
 \ben\label{thm-BKa-dual-norm}
    \begin{split}
    &\|A\|_{\calL(M^p_{\al}(\rd), M^{\i}_{\al}(\rd))}
    \\
    \sim &
    \sup\limits_{k,k' \in \zd}
	\big\|\b_l^{d/p'}\lan K_A(z',z),  
	\check{\va_k^{\al}}(z'-\b_k k')\otimes\check{\eta_l^{\al}}(\b_l l'-z)  \ran\big\}_{l, l'\in \zd}\big\|_{l^{p'}(\zdd)}.
    \end{split}
    \een
    where the implicit constant is independent of $\la$.
\end{theorem}

\begin{proof}
    Recall $A\in \calL(\calS, \calS')$.
    Define $\widetilde{A}\in \calL(\calS, \calS')$ by $\langle \wt{A}g,f\ran:= \overline{\lan Af, g\ran}$
    for all $f,g \in \calS$.
    By Proposition \ref{pp-dual-bd} and Theorem \ref{thm-BKa}, we obtain that $(1)$ is equivalent to $\wt{A}\in \calL(M^1_{\al}, M^{p'}_{\al})$,
    satisfying the following estimate
    \ben\label{thm-BKa-dual-pf1}
    \|A\|_{\calL(M^p_{\al},M^{\fy}_{\al})}
    \sim 
    \|\widetilde{A}\|_{\calL(M^1_{\al}, M^{p'}_{\al})}
    \sim
    \sup_{k,k'}\| \wt{A}(T_{\b_kk'}\check{\vph_k^{\a}})\|_{M^{p'}_{\a}}<\fy.
    \een
    Notice that 
    $\text{supp}\calF(\Box_l^{\a}\wt{A}(T_{\b_kk'}\check{\vph_k^{\a}})) 
    \subset 
    B(\lan l \ran^{\frac{\a}{1-\a}}l, C\lan l \ran^{\frac{\a}{1-\a}})$.
    From Proposition \ref{pp-Fourierexp-alpha}, as in \eqref{thm-BKa-pf3}, we have 
    \ben\label{thm-BKa-dual-pf2}
        \|\wt{A}(T_y\check{\va_k^{\al}})\|_{M^{p'}_{\al}(\rd)}
        \sim 
		\big\|\{ \b_l^{d/{p'}}\Box_l^{\al}(\wt{A}(T_y\check{\va_k^{\al}}))(\b_l l')  \}_{l,l'\in \zd}\big\|_{l^{p'}(\bbZ^{2d})}.
    \een
	And by the definition of $A, \wt{A}$, we have 
	\be
	\begin{split}
		\Box_l^{\al}(\wt{A}(T_y\check{\va_k^{\al}}))(x)
		= &
		\lan \wt{A}(T_y\check{\va_k^{\al}}), \overline{\check{\eta_l^{\al}}(x-\cdot)} \ran
        =
        \overline{ \lan  A\overline{\check{\eta_l^{\al}}(x-\cdot)}, T_y\check{\va_k^{\al}}\ran }
		\\
		= &
		\overline{ \lan K_A(z',z), T_y\check{\va_k^{\al}}(z') \otimes \check{\eta_l^{\al}}(x-z) \ran},
        \\
		= &
		\overline{ \lan K_A(z',z), \check{\va_k^{\al}}(z'-y)\otimes\check{\eta_l^{\al}}(x-z)  \ran}.
	\end{split}
    \ee
    Substituting the above calculation into \eqref{thm-BKa-dual-pf1} and \eqref{thm-BKa-dual-pf2}, we obtain that the statement $(1)$ is equivalent to
    \begin{equation*}
        \sup_{k,k'}\big\| 
         \{ 
           \b_l^{d/{p'}} \lan K_A(z',z), \check{\va_k^{\al}}(z'-\b_kk')\otimes\check{\eta_l^{\al}}(\b_ll'-z)  \ran 
         \}_{l,l'\in \zd}
        \big\|_{l^{p'}(\bbZ^{2d})}
        <\i.
    \end{equation*}
    The norm estimate is also obtained.
\end{proof}

\section{Compactness kernel theorem on $\alpha$-modulation spaces}
First, We establish an expansion for the elements in $\a$-modulation spaces first.
The function sequence $\{\va_k^{\a}\}_{k\in \zd}$ is defined in \eqref{dual-frame-alpha}.
\begin{proposition}\label{pp-exp-alpha}
	Let $p\in (0,\fy]$, $\la\in (0,1]$. Denote $\b_k=\frac{\la}{2rC\lan k\ran^{\frac{\al}{1-\al}}}$.
 For $f\in \mpad$, we have the expansion
	\ben\label{lm-exp-alpha-cd0}
	f=\sum_{k,k'\in \zd}\b_k^d\Box_k^{\al}f(\b_k k')T_{\b_k k'}\check{\va_k^{\al}}
	\een
	with unconditional convergence in $\mpad$ for $p<\fy$, and weak-star convergence in $M^{\fy}_{\a}(\rd)$. Moreover, for every fixed $p\in (0,\fy]$, we have the equivalent relation
	\ben\label{lm-exp-alpha-cd1}
	\|f\|_{\mpad}\sim \big\|\{\b_k^{d/p}\Box_k^{\al}f(\b_k k')\}_{k',k\in \zd}\big\|_{l^p(\zdd)},
	\een
    where the implicit constant is independent of $f$ and $\la$.
	Conversely, if $g\in \calS'(\rd)$ satisfies $\big\|\{\b_k^{d/p}\Box_k^{\al}g(\b_k k')\}_{k',k\in \zd}\big\|_{l^p(\zdd)}<\fy$, 
	we have $g \in \mpad$.
\end{proposition}
\begin{proof}
	When $p<\fy$, for $f\in \mpad$, we have the expansion
	\be
	f=\sum_{k\in \zd}\Box_k^{\al}f
	\ee
	with unconditional convergence in $\mpad$.
	In fact,
	for any fixed constant $N\in \bbN$ and all finite subset $E\subset \zd$ containing $\{k\in \zd: |k|\leq N\}$ we have
	\be
	\big\|f-\sum_{k\in E}\Box_k^{\al}f\big\|_{\mpad}
	=
	\big\|\sum_{k\in E^c}\Box_k^{\al}f\big\|_{\mpad}
	\lesssim 
	\big(\sum_{|k|>N}\|\Box_k^{\al}f\|_{L^p(\rd)}^p\big)^{1/p}
	<\ep_N,
	\ee
	where $\ep_N\rightarrow 0$ as $N\rightarrow \fy$.
	Next, for all $|k|\leq N$, using Proposition \ref{pp-Fourierexp-alpha}, we have the expansion
	\be
	\Box_k^{\al}f=\b_k^d\sum_{k'\in \zd}\Box_k^{\al}f(\b_k k')T_{\b_k k'}\check{\va_k^{\al}}
	\ee
	with unconditional convergence in $L^p(\rd)$, and that
	\ben\label{pf-box-alpha-norm}
	\|\Box_k^{\al}f\|_{L^p(\rd)}
	\sim 
	\b_k^{d/p}\|\{\Box_k^{\al}f(\b_k k')\}_{k'\in \zd}\|_{l^p(\zd)}.
	\een
	Thus, for any constant $M\in \bbN$ and
	all finite subset $F\subset \zd$ containing $\{k'\in \zd: |k'|\leq M\}$ we have
	\be
	\big\|\Box_k^{\al}f-\b_k^d\sum_{k'\in F}\Box_k^{\al}f(\b_k k')T_{\b_k k'}\check{\va_k^{\al}}\big\|_{L^p(\rd)}<\ep_{N,M}.
	\ee
	for all $|k|\leq N$, where $\ep_{N,M}\rightarrow 0$ as $M\rightarrow \fy$.
	For any finite subset $\Om\subset \zdd$ containing $\{(k,k')\in \zd: |k|\leq N, |k'|\leq M\}$, we write
	\be
	\begin{split}
		&\big\|f-\sum_{(k,k')\in \Om}\b_k^d\Box_k^{\al}f(\b_k k')T_{\b_k k'}\check{\va_k^{\al}}\big\|_{\mpad}
		\\
		\lesssim &
		\big\|f-\sum_{|k|\leq N}\Box_k^{\al}f\big\|_{\mpad}
		+
		\big\|\sum_{(k,k')\in \Om, |k|> N}\b_k^d\Box_k^{\al}f(\b_k k')T_{\b_k k'}\check{\va_k^{\al}}\big\|_{\mpad}
        \\
        &+
        \big\|\sum_{|k|\leq N}\Box_k^{\al}f-\sum_{(k,k')\in \Om, |k|\leq N}\b_k^d\Box_k^{\al}f(\b_k k')T_{\b_k k'}\check{\va_k^{\al}}\big\|_{\mpad}
        =:I+II+III.
	\end{split}
	\ee
	Note that $I\lesssim \ep_N$. Using the derivative assumption of $\va_k^{\al}$ and 
 the estimate \eqref{pp-Fourierexp-alpha-cd2}, for any $|k|>N$ and $\Om_k=\{k'\in \zd, (k,k')\in \Om\}$, we have
	\be
	\begin{split}
		&\big\|\sum_{k'\in \Om_k}\b_k^d\Box_k^{\al}f(\b_k k')T_{\b_k k'}\check{\va_k^{\al}}\big\|_{\lpd}
		\lesssim
		\b_k^{d/p}\big\|\{\Box_k^{\al}f(\b_k k')\}_{k' \in \Om_k}\big\|_{l^p(\zd)}
        \lesssim 
        \|\bka f\|_{\lpd},
	\end{split}
	\ee
    where the implicit constant is independent of $k$.
	From this, we conclude that
	\be
	\begin{split}
		II
		\lesssim &
		\big\|\big\{\big\|\sum_{k'\in \Om_k}\b_k^d\Box_k^{\al}f(\b_k k')T_{\b_k k'}\check{\va_k^{\al}}\big\|_{\lpd}\big\}_{|k|> N}\big\|_{l^p(\zd)}
		\\
		\lesssim &
		\big(\sum_{|k|>N}\|\Box_k^{\al}f\|_{L^p(\rd)}^p\big)^{1/p}
		<\ep_N.
	\end{split}
	\ee
	For the estimate of $III$, we have
	\be
	\begin{split}
		III
		= &
		\big\|\sum_{|k|\leq N}\big(\Box_k^{\al}f-\sum_{k'\in \Om_k}\b_k^d\Box_k^{\al}f(\b_k k')T_{\b_k k'}\check{\va_k^{\al}}\big)\big\|_{\mpad}
		\\
		\lesssim &
		\big(\sum_{|k|\leq N}\big\|\Box_k^{\al}f-\sum_{k'\in \Om_k}\b_k^d\Box_k^{\al}f(\b_k k')T_{\b_k k'}\check{\va_k^{\al}}\big\|_{\lpd}^p\big)^{1/p}
		\lesssim 
		N^{d/p}\ep_{N,M}.
	\end{split}
	\ee
	Combining the above estimates of the terms $I$, $II$ and $III$, we obtain
	\be
	\|f-\sum_{(k,k')\in \Om}\b_k^d\Box_k^{\al}f(\b_k k')T_{\b_k k'}\check{\va_k^{\al}}\|_{\mpad}
	\lesssim \ep_N+N^{d/p}\ep_{N,M}
	\ee
	 for any subset $\Om\subset \zdd$ containing $\{(k,k')\in \zd: |k|\leq N, |k'|\leq M\}$.
	 Then, the desired unconditional convergence follows by first taking sufficiently large $N$ and then letting $M$ be sufficiently large.
	  And \eqref{lm-exp-alpha-cd1} follows by \eqref{pf-box-alpha-norm}:
	 \[
	 \|f\|_{\mpad}
	 =    \|\{\|\Box_k^{\al}f\|_{L^p(\bbR^d)}\}_{k\in\zd}\|_{l^p(\zd)}
	 \sim \big(\|\{\b_k^{d/p}\Box_k^{\al}f(\b_k k')\}_{k',k\in \zd}\big\|_{l^p(\zdd)}.
	 \]
	 When $p=\i$, the conclusion can be verified by a similar argument with suitable modification.
	 
	 Finally, we turn to the inverse direction. Note that $\bka g\in \calS'_{B(\langle
	 	k\rangle^{\frac{\alpha}{1-\alpha}}k, C\langle
	 	k\rangle^{\frac{\alpha}{1-\alpha}})}$.
	 	Using Proposition \ref{pp-Fourierexp-alpha}, we find that $\bka g\in L^p$ with the norm estimates:
	 	\be
	 	\|\bka g\|_{\lpd}
	 	\sim 
	 	\big\|\{\b_k^{d/p}\Box_k^{\al}g(\b_k k')\}_{k'\in \zd}\big\|_{l^p(\zd)}<\fy.
	 	\ee
	 	Then the desired conclusion follows by the definition of $\al$-modulation space.
\end{proof}

\begin{proposition}\label{pp-tbs}
	Let $p\in (0,\fy]$. Suppose that $\{\psi_k: k\in \zn\}$ is a totally bounded subset of $\mpad$. 
	Then, for any bounded subset of $l^{p\wedge 1}(\zn)$ denoted by $B$,
	the following set
	\be
	\Psi=\{f=\sum_{k\in \zn}\la_k\psi_k:\ \{\la_k\}_{k\in \zn}\in B\}
	\ee
	is also totally bounded in $\mpad$.
\end{proposition}
\begin{proof}
	Without loss of generality, we assume that $B=\big\{\{\la_k\}_{k\in \zn}: \|\{\la_k\}_{k\in \zn}\|_{l^{p\wedge 1}(\zn)}\leq 1\big\}$.
	For any fixed small constant $\ep>0$, since $\{\psi_k\}_{k\in \zn}$ is totally bounded in $\mpad$, there exist a finite $\ep$-net denoted by
	$
	H_{\ep}=\{h_j\}_{j=1}^N
	$
	such that for any $k\in \zn$, there exists a $j\in \{1,2,\cdots,N\}$ such that $\|\psi_k-h_j\|_{\mpad}<\ep$. 
	Such a relationship leads to a (possibly non-unique) mapping $\r$ from $\zn$ into $\{1,2,\cdots,N\}$ such that
	\be
	\|\psi_k-h_{\r(k)}\|_{\mpad}<\ep.
	\ee
	Using this map, we have a decomposition of $\zn$
	\be
	\zn=\bigcup_{j=1}^d\r^{-1}(j).
	\ee
	Write
	\be
	\begin{split}
			\sum_{k\in \zn}\la_k\psi_k
		= &
		\sum_{j=1}^N\sum_{k\in \r^{-1}(j)}\la_k\psi_k
		\\
		= &
		\sum_{j=1}^N\sum_{k\in \r^{-1}(j)}\la_k(\psi_k-h_j)+\sum_{j=1}^N\big(\sum_{k\in \r^{-1}(j)}\la_k\big)h_j.
	\end{split}
	\ee
	Note that
	\be
	\begin{split}
		\big\|\sum_{j=1}^N\sum_{k\in \r^{-1}(j)}\la_k(\psi_k-h_j)\big\|_{\mpad}
		\leq &
		\big(\sum_{j=1}^N\sum_{k\in \r^{-1}(j)}\big\|\la_k(\psi_k-h_j)\big\|_{\mpad}^{p\wedge 1}\big)^{1/p\wedge 1}
		\\
		< &
		\ep\big(\sum_{j=1}^N\sum_{k\in \r^{-1}(j)}|\la_k|^{p\wedge 1}\big)^{1/p\wedge 1}\leq \ep.
	\end{split}
	\ee
	On the other hand, 
	since
	\be
	\bigg|\sum_{k\in \r^{-1}(j)}\la_k\bigg|\leq \|(\la_k)_{k\in \zn}\|_{l^{p\wedge 1}(\zn)}\leq 1,
	\ee
    then the subset 
	\be
	\bigg\{\sum_{j=1}^N\big(\sum_{k\in \r^{-1}(j)}\la_k\big)h_j: (\la_k)_{k\in \zn}\in B\bigg\}
	\ee
	has a finite $\epsilon$ net in $\mpad$. The above arguments yield that $\Psi$ has a finite $2^{\frac{1}{p\wedge 1}}\ep$-net. For the arbitrary of $\ep$, we obtain 
	the desired conclusion that $\Psi$ is a totally bounded subset of $\mpad$.
\end{proof}

Before proceeding with our main thread, we give the following two characterizations of the totally bounded subsets in $l^p$ and $M^p_{\al}$, respectively.
We omit the proof for the $l^p$ case, as it can be verified directly by the fact that any bounded set is totally bounded in metric space with finite dimension.
\begin{lemma}\label{lm-lp-tot.bdd}
	Let $p\in (0,\fy)$.
	A subset $\calF\subset l^p(\zd)$ is totally bounded if and only if the following two statements hold:
	\bn
	\item Uniform boundedness: 
	\ben\label{lm-lp-u.bdd}
	\sup_{\vec{a}\in \calF}\|\vec{a}\|_{l^p}<\fy;
	\een
	\item Uniform disapperance: 
	\ben\label{lm-lp-u.disa}
	\lim_{N\rightarrow \fy}\sup_{\vec{a}\in \calF}
	\big\|\{a_k\}_{k\in \zd\bs [-N,N]^{d}}\big\|_{l^p(\zd)}=0.
	\een
	\en
\end{lemma}
%
\begin{theorem}\label{thm-tbsa}
	Let $p\in (0,\fy)$, and $\calF$ be a subset of $\mpad$.
	The following three statements are equivalent: 
	\bn
	\item $\calF$ is totally bounded;
	\item $\calF$ satisfies the properties of uniform boundedness and disappearance:
	 \bn
	 \item Uniform boundedness: 
	 \ben\label{thm-tbsa-cd1}
	 \sup_{f\in \calF}\|f\|_{\mpad}<\fy;
	 \een
	 \item Uniform disapperance: 
	 \ben\label{thm-tbsa-cd2}
	 \lim_{N\rightarrow \fy}\sup_{f\in \calF}
	 \big\|\{\lan k\ran^{\frac{-d\al}{(1-\al)p}}\Box_k^{\al}f(\b_k k')\}_{(k,k')\in \zdd\bs [-N,N]^{2d}}\big\|_{l^p(\zdd)}=0.
	 \een
	 \en
	\item $\wt{\calF}:=\Big\{\big\{\lan k\ran^{\frac{-d\al}{(1-\al)p}} \Box_k^{\al}f(\b_k k')\big\}_{(k,k')\in \zdd}: f\in \calF\Big\}$ is totally bounded in $l^p(\zdd)$.
	\en
\end{theorem}
\begin{proof}
	$(1)\Longrightarrow(2):$
    The statement $(a)$ follows directly by the fact that a totally bounded subset is bounded.
	To verify the statement $(b)$, note that it is valid for every fixed $f\in \mp$, that is,
	\ben\label{thm-tbsa-pf1}
	\lim_{N\rightarrow \fy}\|\{\lan k\ran^{\frac{-d\al}{(1-\al)p}}\Box_k^{\al}f(\b_k k')\}_{(k,k')\in \zdd\bs [-N,N]^{2d}}\big\|_{l^p(\zdd)}=0,\een
	where we use the equivalent $\mpad$-norm in Proposition \ref{pp-exp-alpha}.
	Thus, \eqref{thm-tbsa-pf1} is also valid uniformly for the finite $\ep$-net $\{h_j\}_{j=1}^N$ of $\calF$. 
	Furthermore, \eqref{thm-tbsa-pf1} is valid uniformly for all the elements $f$ in $\calF$ by the following estimates:
	\be
	\begin{split}
		&\big\|\{\lan k\ran^{\frac{-d\al}{(1-\al)p}}\Box_k^{\al}f(\b_k k')\}_{(k,k')\in \zdd\bs [-N,N]^{2d}}\big\|_{l^p(\zdd)}
		\\
		\lesssim &
		\big\|\{\lan k\ran^{\frac{-d\al}{(1-\al)p}}\Box_k^{\al}h_j(\b_k k')\}_{(k,k')\in \zdd\bs [-N,N]^{2d}}\big\|_{l^p(\zdd)}
		\\
		&+
		\big\|\{\lan k\ran^{\frac{-d\al}{(1-\al)p}}\Box_k^{\al}(f-h_j)(\b_k k')\}_{(k,k')\in \zdd\bs [-N,N]^{2d}}\big\|_{l^p(\zdd)}
		\\
		\lesssim &
		\big\|\{\lan k\ran^{\frac{-d\al}{(1-\al)p}}\Box_k^{\al}h_j(\b_k k')\}_{(k,k')\in \zdd\bs [-N,N]^{2d}}\big\|_{l^p(\zdd)}
		+\|f-h_j\|_{\mpad}
		\\
		\leq &
		\big\|\{\lan k\ran^{\frac{-d\al}{(1-\al)p}}\Box_k^{\al}h_j(\b_k k')\}_{(k,k')\in \zdd\bs [-N,N]^{2d}}\big\|_{l^p(\zdd)}
		+\ep.
	\end{split}
	\ee
	
	$(2)\Longrightarrow(1):$
	First, observe that $\{\lan k\ran^{\frac{\al d}{1-\al}(1/p-1)}T_{\b_k k'}\check{\va_k^{\al}}: (k,k')\in \zdd\cap [-N,N]^{2d}\}$ is totally bounded in $\mpad$.
    Denote by
	\be
	\calF_N=\bigg\{\sum_{(k,k')\in \zdd\cap [-N,N]^{2d}}\b_k^d\Box_k^{\al}f(\b_k k')T_{\b_k k'}\check{\va_k^{\al}}: f\in \calF\bigg\}.
	\ee
	It follows by $(a)$, \eqref{lm-exp-alpha-cd1} and Proposition \ref{pp-tbs} that $\calF_N$ is totally bounded in $\mpad$.
	Moreover, by using \eqref{pp-Fourierexp-alpha-cd2},
    the distance between $\calF$ and $\calF_N$ is bounded by
	\be
	\begin{split}
			&\sup_{f\in \calF}\bigg\|f-\sum_{(k,k')\in \zdd\cap [-N,N]^{2d}}\b_k^d\Box_k^{\al}f(\b_k k')T_{\b_k k'}\check{\va_k^{\al}}\bigg\|_{\mpad}
			\\
			= &
			\sup_{f\in \calF}\bigg\|\sum_{(k,k')\in \zdd\bs [-N,N]^{2d}}\b_k^d\Box_k^{\al}f(\b_k k')T_{\b_k k'}\check{\va_k^{\al}}\bigg\|_{\mpad}
            \\
			\lesssim &
			\sup_{f\in \calF}
            \bigg(\sum_{k\in \zdd\cap [-N,N]^{d}}
            \bigg(\bigg\|\sum_{k'\in \zdd\bs [-N,N]^{d}}\b_k^d\Box_k^{\al}f(\b_k k')T_{\b_k k'}\check{\va_k^{\al}}\bigg\|^p_{L^p}\bigg)^{1/p}
            \\
            + &
            \sum_{k\in \zdd\bs [-N,N]^{d}}
            \bigg(\bigg\|\sum_{k'\in \zdd}\b_k^d\Box_k^{\al}f(\b_k k')T_{\b_k k'}\check{\va_k^{\al}}\bigg\|^p_{L^p}\bigg)^{1/p}\bigg)
			\\
			\lesssim &
			\sup_{f\in \calF}\big\|\{\lan k\ran^{\frac{-d\al}{(1-\al)p}}\Box_k^{\al}f(\b_k k')\}_{(k,k')\in \zdd\bs [-N,N]^{2d}}\big\|_{l^p(\zdd)},
	\end{split}
	\ee
	which tends to zero as $N\rightarrow \fy$.
	The above arguments yield that $\calF$ is a totally bounded subset of $\mpad$.
	
	$(2)\Longleftrightarrow(3):$ 
	It follows by \eqref{lm-exp-alpha-cd1} that $(a)$ is equivalent to 
	\ben\label{pf-mp-norm-discrete}
	\sup_{f\in \calF}
	\big\|\{\lan k\ran^{\frac{-d\al}{(1-\al)p}}\Box_k^{\al}f(\b_k k')\}_{(k,k')\in \zdd}\big\|_{l^p(\zdd)}<\i.
	\een
	From this, the condition $(b)$ and Lemma \ref{lm-lp-tot.bdd}, we conclude that (2) is equivalent to that 
    $\wt{\calF}$ is a totally bounded subset of $l^p(\zdd)$.
\end{proof}

By using the viewpoint that the action of the linear operator $A$ on the function space can be reduced to its action
on the atoms, we can deduce the following characterization for the compactness of a linear operator on $\al$-modulation spaces.

\begin{theorem}[Compactness kernel theorem for $\alpha$-modulation spaces]\label{thm-CKa}
	Suppose that $p,q\in (0,\fy)$ satisfying $p\leq q\wedge 1$.
	Let $A\in \calL(\calS(\rd), \calS'(\rd))$ be a linear operator with
	distributional kernel $K_A$.	
	For any $\la\in (0, 1]$, denote $\b_k=\frac{\la}{2rC\lan k\ran^{\frac{\al}{1-\al}}}$ for all $k\in \zd$.
	The following statements are equivalent.
	\bn
	\item
	$A\in\calK(M^p_{\al}(\rd), M^q_{\al}(\rd))$;
	\item
	$\{\lan k\ran^{\frac{\al d}{1-\al}(1/p-1)}A(T_{\b_k k'} \check{\va_k^{\al}}): k,k' \in \zd\}$
	is a totally bounded subset of $M^{q}_{\al}(\rd)$;
	\item
	The following two properties are valid:
	
	(a) Uniform boundedness:
	\be
	\sup\limits_{k,k' \in \zd}\lan k\ran^{\frac{\al d}{1-\al}(1/p-1)}
	\big\|\big\{\lan l\ran^{\frac{-\al d}{(1-\al)q}}\lan K_A(z',z),  \overline{\check{\eta_l^{\al}}(\b_l l'-z')\otimes\check{\va_k^{\al}}(z-\b_k k')}  \ran\big\}_{l, l'\in \zd}\big\|_{l^q(\zdd)}<\fy,
	\ee
	
	(b) Uniform disappearance:
	\be
	\lim\limits_{N\rightarrow \fy}\sup\limits_{k,k' \in \zd}\lan k\ran^{\frac{\al d}{1-\al}(1/p-1)}
	\big\|\big\{\lan l\ran^{\frac{-\al d}{(1-\al)q}}\lan K_A(z',z),  \overline{\check{\eta_l^{\al}}(\b_l l'-z')\otimes\check{\va_k^{\al}}(z-\b_k k')}  \ran\big\}_{(l, l')\in \zdd\bs [-N,N]^{2d}}\big\|_{l^q(\zdd)}=0.
	\ee
	\item
	$\Big\{ \big\{\lan k\ran^{\frac{\al d}{1-\al}(1/p-1)}\lan l\ran^{\frac{-\al d}{(1-\al)q}}\lan K_A(z',z),  \overline{\check{\eta_l^{\al}}(\b_l l'-z')\otimes\check{\va_k^{\al}}(z-\b_k k')}  \ran\big\}_{l, l'\in \zd}: k,k' \in \zd\big\}$
	is a totally bounded subset of $l^q(\zdd)$.
	\en
\end{theorem}
\begin{proof}
	First, the relation $(1)\Longrightarrow (2)$ follows by $A\in \calK(\mpad,\mqad)$
	and the fact that
	\be
	\{\lan k\ran^{\frac{\al d}{1-\al}(1/p-1)}T_y\check{\va_k^{\al}}\}_{k\in \zd, y\in \rd}
	\ee
	is a bounded subset of $\mpad$.
	
	Next, we turn to the relation $(2)\Longrightarrow (1)$.
	Note that the statement (2) implies that $\{\lan k\ran^{\frac{\al d}{1-\al}(1/p-1)}A(T_{\b_k k'} \check{\va_k^{\al}})\}_{k,k' \in \zd}$
	is a bounded subset of $M^{q}_{\al}(\rd)$.
	Using this and Theorem \ref{thm-BKa}, we have $A\in \calL(\mpad,\mqad)$.
	Denote by $\calF$ a bounded subset of $\mpad$. For $f\in\calF$, write
	\be
	f=\sum_{k,k'\in \zd}\b_k^d\Box_k^{\al}f(\b_k k')T_{\b_k k'}\check{\va_k^{\al}}
	\ee
	with unconditional convergence in $\mpad$.
	Applying the operator $A$ on both sides of the above expansion, we obtain
		\be
		\begin{split}
			Af
			= &
			\sum_{k,k'\in \zd}\b_k^d\Box_k^{\al}f(\b_k k')A(T_{\b_k k'}\check{\va_k^{\al}})
			\\
			= &
			\sum_{k,k'\in \zd}
			\b_k^{d/p}\Box_k^{\al}f(\b_k k') \b_k^{d(1-1/p)}A(T_{\b_k k'}\check{\va_k^{\al}})
		\end{split}
	\ee
	with unconditional convergence in $\mqad$.
	Note that 
	\be
	\begin{split}
		\big\|\{\lan k\ran^{\frac{-\al d}{(1-\al)p}}\Box_k^{\al}f(\b_k k')\}_{k',k\in \zd}\big\|_{l^{q\wedge 1}(\zdd)}
		\lesssim &
		\big\|\{\lan k\ran^{\frac{-\al d}{(1-\al)p}}\Box_k^{\al}f(\b_k k')\}_{k',k\in \zd}\big\|_{l^{p}(\zdd)}
		\\
		\lesssim & \|f\|_{\mpad}\leq \sup_{f\in \calF}\|f\|_{\mpad}.
	\end{split}
	\ee
	Combining this with the fact that $\{\lan k\ran^{\frac{\al d}{1-\al}(1/p-1)}A(T_{\b_k k'} \check{\va_k^{\al}})\}_{k,k' \in \zd}$
	is a totally bounded subset of $M^{q}_{\al}(\rd)$, we use Proposition \ref{pp-tbs} to deduce that 
	$
	\{Af: f\in \calF\}
	$
	is a totally bounded subset of $\mqad$. This implies the desired conclusion $A\in \calK(\mpad,\mqad)$.
	
	Finally, the equivalent relation $(2)\Longleftrightarrow (3)\Longleftrightarrow (4)$ follows by \eqref{lm-exp-alpha-cd1}, Theorem \ref{thm-tbsa} and \eqref{thm-BKa-pf4}. 
\end{proof}

\section{Simplification and complements}
\subsection{The concise form of the kernel theorems}
In Theorems \ref{thm-BKa}, \ref{thm-BKa-dual} and \ref{thm-CKa}, 
the statements are based on the specific expansion of the elements in $\al$-modulation space.
Obviously, the characterization of the linear operator should not based on the specific expansion, but only
depends on the starting and arriving spaces of the linear operator.
Based on this viewpoint, we would like to give the more concise form of Theorems \ref{thm-BKa}, \ref{thm-BKa-dual} and \ref{thm-CKa},
in which
the kernel of linear operator is characterized by the corresponding mixed $\al$-modulation space.

We first introduce the mixed frequency decomposition operators associated with $\rdd$:
\be
\Box_{k,l}^{\al}:= \scrF^{-1}(\eta_k^{\al}\otimes \eta_l^{\al})\scrF,\ \ k,l\in \zd,
\ee
where $\{\eta_k^{\al}\}_{k\in \zd}$ is just the function sequence mentioned in \eqref{pre-2}.
The corresponding mixed $\al$-modulation space can be naturally defined as follows.
\begin{definition}[Mixed $\al$-modulation space]\label{def-mix-alpha}
    Let $0<p_1,p_2,q_1,q_2\leq \fy$, $s, t\in \rr$.
    The mixed $\al$-modulation space $M^{p_1,p_2,q_1,q_2}_{\al,s,t}$ consists
	of all $F\in \calS'(\rdd)$ such that the (quasi-)norm
	\be
 \begin{split}
     \|F\|_{M^{p_1,p_2,q_1,q_2}_{\al,s,t}}
    = &
    \|\lan k\ran^{\frac{s}{1-\al}}\lan l\ran^{\frac{t}{1-\al}}\Box_{k,l}^{\al}F(x,y)\|_{L^{p_1,p_2,q_1,q_2}_{x,y,k,l}}
 \end{split}
    \ee
	is finite. Here $L^{p_1,p_2,q_1,q_2}_{x,y,k,l}$ denotes the mixed norm space with norm
 \be
\|H(x,y,k,l)\|_{L^{p_1,p_2,q_1,q_2}_{x,y,k,l}}
=\left(\sum_{l\in \zd}\left(\sum_{k\in \zd}\left(\int_{\rd}\left(\int_{\rd}|H(x,y,k,l)|^{p_1}dx\right)^{\frac{p_2}{p_1}}dy\right)^{\frac{q_1}{p_2}}\right)^{\frac{q_2}{q_1}}\right)^{\frac{1}{q_2}}.
 \ee
\end{definition}
\begin{remark}
    Notice that the mixed $\al$-modulation space $M^{p,p,q,q}_{\al,0,0}$ does not coincide with the 
    usual $\al$-modulation space $M^{p,q}_{\al}(\rdd)$ unless $\al=0$. See \cite{Bishop2010JoMAaA} for the definition of mixed
    modulation space, and see \cite{CleanthousGeorgiadis2020TAMS} for another type of mixed $\al$-modulation space.
\end{remark}

To state our kernel theorems we need a further extension of Definition \ref{def-mix-alpha}.
\begin{definition}[Mixed $\al$-modulation space with permutation]\label{def-mix-alpha-permutation}
    Let $0<p_1,p_2,q_1,q_2\leq \fy$, $s,t\in \rr$.
    The  mixed-modulation space $M^{p_1,p_2,q_1,q_2}_{\al,s,t}(c_i) (i=1,2)$ consists
	of all $F\in \calS'(\rdd)$ such that the (quasi-)norm
	\be
    \|F\|_{M^{p_1,p_2,q_1,q_2}_{\al,s,t}(c_1)}=\|\lan k\ran^{\frac{s}{1-\al}}\lan l\ran^{\frac{t}{1-\al}}\Box_{k,l}^{\al}F(x,y)\|_{L^{p_1,p_2,q_1,q_2}_{x,k,y,l}}
    \ee
    or
    \be
    \|F\|_{M^{p_1,p_2,q_1,q_2}_{\al,s,t}(c_2)}=\|\lan k\ran^{\frac{s}{1-\al}}\lan l\ran^{\frac{t}{1-\al}}\Box_{k,l}^{\al}F(x,y)\|_{L^{p_1,p_2,q_1,q_2}_{y,l,x,k}}
    \ee
	is finite. 
 Here $L^{p_1,p_2,q_1,q_2}_{x,k,y,l}$ and $L^{p_1,p_2,q_1,q_2}_{y,l,x,k}$ denote the mixed norm spaces having norms similar to
the norm of $L^{p_1,p_2,q_1,q_2}_{x,y,k,l}$ but with different integration order.
 
\end{definition}
Note that the function spaces $M^{q,q,\fy,\fy}_{\al,0,\a d(1/p-1)}(c_1)$ and $M^{p',p',\fy,\fy}_{\al,0,0 }(c_2)$ are used
in Theorems \ref{thm-sBKa} and \ref{thm-sBKa-dual}.
Of course one can conduct a systematic study on these mixed $\al$-modulation spaces, but this is not the focus of our article.
Here, we only do some necessary research for the needs of our main theorems.

\begin{proof}[Proof of Theorem \ref{thm-sBKa}]
    As in the proof of Theorem \ref{thm-BKa}, the relation $(1)\Longrightarrow (2)$ follows by $A\in \calL(\mpad,\mqad)$
	and the fact that
	\be
	\{\lan k\ran^{\frac{\al d}{1-\al}(1/p-1)}T_y\check{\eta_k^{\al}}\}_{k\in \zd, y\in \rd}
	\ee
	is a bounded subset of $\mpad$. 

    Next, we turn to the proof of $(2)\Longrightarrow (1)$.
    Define
    \be
    \widetilde{\va_k^{\al}}=\sum_{l\in \La_k}\eta_l^{\al},\ \ \ \La_k=\{l\in \zd: \text{supp}\eta_l^{\al}\cap B(\lan k\ran^{\frac{\al}{1-\al}}k, C\lan k\ran^{\frac{\al}{1-\al}})\neq \emptyset\}.
    \ee
    There exists a constant $\widetilde{r}>0$ such that 
    the function sequence $\{\widetilde{\va_k^{\al}}\}_{k\in \zd}$ satisfies the condition \eqref{dual-frame-alpha}.
    Note that $|\La_k|\lesssim 1$ uniformly for all $k\in \zd$. 
    If the statement (2) holds, 
    we conclude that 
    \be
    \{\lan k\ran^{\frac{\al d}{1-\al}(1/p-1)}A(T_{y} \check{\widetilde{\va_k^{\al}}}): k\in \zd, y\in \rd\}
    \ee
	is also a bounded subset of $M^{q}_{\al}(\rd)$.
    From this and Theorem \ref{thm-BKa}, we obtain the desired conclusion in the statement (1).

    Finally, let us deal with the relation $(2)\Longleftrightarrow (3)$.
    Write
	\be
	\begin{split}
		\Box_l^{\al}(A(T_y\check{\eta_k^{\al}}))(x)
        = &
        \check{\eta_l^{\al}}\ast(A(T_y\check{\eta_k^{\al}}))(x)
        \\
		= &
  \lan A(T_y\check{\eta_k^{\al}}), \overline{\check{\eta_l^{\al}}(x-\cdot)} \ran
		\\ 
		= &
  \lan K_A(z',z), \overline{\check{\eta_l^{\al}}(x-z')\otimes \check{\eta_k^{\al}}(z-y)} \ran
  \\
  = &
    \lan K_A(z',z), \overline{\check{\eta_l^{\al}}(x-z')\otimes \check{\eta_{-k}^{\al}}(y-z)} \ran
    =
    \Box_{l,-k}^{\al}K_A(x,y),
	\end{split}
    \ee
    where we use the fact
    \be
    \check{\eta_k^{\al}}(\xi)=\int_{\rd}\eta_k^{\al}(x)e^{2\pi ix\cdot \xi}dx
    =
    \int_{\rd}\eta_{-k}^{\al}(-x)e^{2\pi ix\cdot \xi}dx=\check{\eta_{-k}^{\al}}(-\xi).
    \ee
    A direct calculation yields that
    \be
    \begin{split}
        \sup_{k\in \zd, y\in \rd}\lan k\ran^{\frac{\al d}{1-\al}(1/p-1)}\|A(T_{y} \check{\eta_k^{\al}})\|_{M^q_{\al}}
        = &
        \sup_{k\in \zd, y\in \rd}\lan k\ran^{\frac{\al d}{1-\al}(1/p-1)}
        \|\{\|\Box_l^{\al}(A(T_{y} \check{\eta_k^{\al}}))(x)\|_{L^q_x}\}_{l\in \zd}\|_{l^q}
        \\
        = &
        \sup_{k\in \zd, y\in \rd}\lan k\ran^{\frac{\al d}{1-\al}(1/p-1)}
        \|\{\|\Box_{l,-k}^{\al}K_A(x,y)\|_{L^q_x}\}_{l\in \zd}\|_{l^q}
        \\
        = &
        \sup_{k\in \zd, y\in \rd}\lan k\ran^{\frac{\al d}{1-\al}(1/p-1)}
        \|\{\|\Box_{l,k}^{\al}K_A(x,y)\|_{L^q_x}\}_{l\in \zd}\|_{l^q}
        \\
        = &
        \|K_A\|_{M^{q,q,\fy,\fy}_{\al,0,\al d(1/p-1)}(c_1)}.
    \end{split}
    \ee
\end{proof}

Combining Theorem \ref{thm-sBKa} and a duality argument, we give the proof of Theorem \ref{thm-sBKa-dual}.
\begin{proof}[The proof of Theorem \ref{thm-sBKa-dual}.]
    Define $\widetilde{A}\in \calL(\calS, \calS')$ by $\langle \wt{A}g,f\ran:= \overline{\lan Af, g\ran}$
    for all $f,g \in \calS$.
   It follows by Proposition \ref{pp-dual-bd} and Theorem \ref{thm-BKa} that
    $A\in \calL(M^p_{\al}, M^{\fy}_{\al})$
    is equivalent to $\wt{A}\in \calL(M^1_{\al}, M^{p'}_{\al})$.
    From this and Theorem  \ref{thm-sBKa}, we obtain
    \be
    A\in \calL(M^p_{\al}, M^{\fy}_{\al})
    \Longleftrightarrow 
    \wt{A}\in \calL(M^1_{\al}, M^{p'}_{\al})
    \Longleftrightarrow
    K_{\wt{A}}\in M^{p',p',\fy,\fy}_{\al,0,0}(c_1).
    \ee
    Note that
    \be
    \begin{split}
        \Box_{k,l}^{\a}K_{\wt{A}}(x,y)
  = &
  \lan K_{\wt{A}}(z',z), \overline{\check{\eta_k^{\al}}(x-z')\otimes \check{\eta_l^{\al}}(y-z)}\ran
  \\
  = &
  \overline{\lan K_{A}(z',z), \check{\eta_l^{\al}}(y-z')\otimes \check{\eta_k^{\al}}(x-z)\ran} 
  \\
  = &
  \overline{\lan K_{A}(z',z), \overline{\check{\eta_{-l}^{\al}}(y-z')\otimes \check{\eta_{-k}^{\al}}(x-z)}\ran} 
  =\overline{\Box_{-l,-k}K_{A}(y,x)}.
    \end{split}
    \ee
We have
\be
\begin{split}
    \|K_{\wt{A}}\|_{M^{p',p',\fy,\fy}_{\al,0,0}(c_1)}
    = &
    \sup_{l,y}\|\{\|\Box_{k,l}^{\a}K_{\wt{A}}(x,y)\|_{L^{p'}_x}\}_k\|_{l^{p'}}
    \\
    = &
    \sup_{l,y}\|\{\|\overline{\Box_{-l,-k}^{\a}K_{A}(y,x)}\|_{L^{p'}_x}\}_k\|_{l^{p'}}
    \\
    = &
    \sup_{l,y}\|\{\|\Box_{l,k}^{\a}K_{A}(y,x)\|_{L^{p'}_x}\}_k\|_{l^{p'}}
    =\|K_A\|_{M^{p',p',\fy,\fy}_{\al,0,0}(c_2)}.
\end{split}
\ee
\end{proof}

To deal with the compactness case, we introduce the subspace of $M^{q,q,\fy,\fy}_{\al,0,\al d(1/p-1)}(c_1)$.

\begin{definition}
    Let $q\in (0,\fy)$. 
    The function space $\tilde{M}^{q,q,\fy,\fy}_{\al,0,\al d(1/p-1)}(c_1)$ consists of all the elements in $M^{q,q,\fy,\fy}_{\al,0,\al d(1/p-1)}(c_1)$ such that
\be
	\lim\limits_{N\rightarrow \fy}\sup\limits_{k\in \zd,y\in \rd}\lan k\ran^{\frac{\al d}{1-\al}(1/p-1)}
	\big\|\big\{\lan l\ran^{\frac{-\al d}{(1-\al)q}}\Box_{l,k}K_A(\b_ll',y)\big\}_{(l, l')\in \zdd\bs [-N,N]^{2d}}\big\|_{l^q(\zdd)}=0.
	\ee
\end{definition}

\begin{proof}[Proof of Theorem \ref{thm-sCKa}]
    As in the proof of Theorem \ref{thm-CKa}, the relation $(1)\Longrightarrow (2)$ follows by $A\in \calK(\mpad,\mqad)$
	and the fact that
	\be
	\{\lan k\ran^{\frac{\al d}{1-\al}(1/p-1)}T_y\check{\eta_k^{\al}}\}_{k\in \zd, y\in \rd}
	\ee
	is a bounded subset of $\mpad$. 

    Next, we turn to the proof of $(2)\Longrightarrow (1)$.
    Denote by $\{\widetilde{\va_k^{\al}}\}_{k\in \zd}$ the function sequence defined in the proof of Theorem \ref{thm-sBKa}.
    If the statement (2) holds, 
    we conclude that 
    \be
    \{\lan k\ran^{\frac{\al d}{1-\al}(1/p-1)}A(T_{y} \check{\widetilde{\va_k^{\al}}}): k\in \zd, y\in \rd\}
    \ee
	is also a totally bounded subset of $M^{q}_{\al}(\rd)$.
    From this and Theorem \ref{thm-CKa}, we obtain the desired conclusion in the statement (1).

    Finally, the relation $(2)\Longleftrightarrow (3)$ follows by Theorem \ref{thm-tbsa} and the definition of $\tilde{M}^{q,q,\fy,\fy}_{\al,0,\al d(1/p-1)}(c_1)$.
\end{proof}    

\begin{remark}
    As a direct conclusion of the fact that $\calK(M^p_{\al}(\rd), M^q_{\al}(\rd))$ is a closed subspace of 
    $\calL(M^p_{\al}(\rd), M^q_{\al}(\rd))$,
    the function space $\tilde{M}^{q,q,\fy,\fy}_{\al,0,\al d(1/p-1)}(c_1)$
    is also a closed subspace of $M^{q,q,\fy,\fy}_{\al,0,\al d(1/p-1)}(c_1)$.
\end{remark}

\subsection{The case of modulation space}
Recall that the space $\mpad$ coincides with the modulation space $\mpd$ when $\al=0$.
Thus, the kernel theorems for modulation space can be automatically implied by our general conclusions
for $\al$-modulation space.
Here, we would like to give an independent proof for the modulation case, by using the Gabor characterization. 

For $z=(x,\xi)\in \rdd$, modulation operator $M_{\xi}$ and time-frequency shift $\pi(z)$ are defined, respectively, by
\be
 M_{\xi}f(t)=e^{2\pi it\cdot \xi}f(t),\ \ \ \  \pi(z)f(t)=M_{\xi}T_x  f(t)=e^{2\pi it\cdot\xi}f(t-x).
\ee
For a fixed window function $g\in \calS(\rd)\bs\{0\}$, there exists a lattice $\La=\zdd/2^N$ with a large constant $N\in \bbN$, such that
$\calG(g,\La)=\{\pi(\la)g:\ \la\in \La\}$ is a frame on $L^2(\rd)$ with the canonical dual frame $\calG(\g,\La)$.
We recall the Gabor characterization of modulation space $\mpd$, 
see \cite[Corollary 12.2.6]{GrochenigBook2013} and \cite[Theorem 3.7]{GalperinSamarah2004ACHA}.

\begin{proposition}[Gabor characterization of modulation spaces]\label{pp-Gabor-mod}
    Suppose that $p\in (0,\fy]$.
    Let $g\in \calS(\rd)$ such that $\calG(g,\La)$ is a frame on $L^2(\rd)$ with the canonical dual frame $\calG(\g,\La)$.
    For $f\in \calS'(\rd)$, the following two statements are equivalent
    \bn
     \item 
     $f\in \mpd$;
     \item
     $\{\lan f, \pi(\la)\g\ran\}_{\la\in \La}\in l^p$;
     \item
     $\{\lan f, \pi(\la)g\ran\}_{\la\in \La}\in l^p$.
    \en
    Moreover, if one of the above statements holds, we have the norm estimate
    \be
    \|f\|_{\mpd}
    \sim \|\{\lan f, \pi(\la)\g\ran\}_{\la\in \La}\|_{l^p}
    \sim \|\{\lan f, \pi(\la)g\ran\}_{\la\in \La}\|_{l^p},
    \ee 
    and the Gabor expansion
      \be
  f=\sum_{\la \in \La}\langle f, \pi(\la)\g\rangle \pi(\la)g
   =\sum_{\la \in \La}\langle f, \pi(\la)g\rangle \pi(\la)\g
  \ee
  with unconditional convergence in $\mpd$ if $p<\fy$, and weak-star convergence in $M^{\fy}$ otherwise.
\end{proposition}

The short-time Fourier transform (STFT) of $f\in \calS'$ with respect to a window $g\in \calS\bs\{0\}$ is defined by
\be
V_gf(x,\xi):=\lan f,\pi(z)g\ran_{\calS',\calS}.
\ee
We recall the equivalent norm of modulation space
\be
\|f\|_{\mpd}\sim \|V_gf\|_{L^p(\rdd)}.
\ee
See \cite{Triebel1983ZFA} for more details for the equivalent norm of modulation space.

\begin{theorem}[Boundedness kernel theorem for modulation spaces]\label{thm-BK}
Let $p,q\in (0,\fy]$ satisfying $p\leq q\wedge 1$.
Let $g\in \calS(\rd)$ such that $\calG(g,\La)$ is a frame on $L^2(\rd)$ with the canonical dual frame $\calG(\g,\La)$.
	Let $A\in \calL(\calS(\rd), \calS'(\rd))$ be a linear operator with
	distributional kernel $K_A$.	
	The following statements are equivalent.
	\bn
	\item
	$A\in\calL(M^p(\rd), M^q(\rd))$;
	\item
	$\{A(\pi(z) g): z \in \rdd\}$
	is a bounded subset of $M^{q}(\rd)$;
	\item
	$\{A(\pi(\la) g): \la \in \La\}$
	is a bounded subset of $M^{q}(\rd)$;
	\item 
	$\sup_{z\in \rdd}\big\|V_{g\otimes \bar{g}}K_A(z_1',z_1,z_2',-z_2)\big\|_{L^q_{z'}(\rdd)}<\fy$;
 	\item 
	$\sup_{\la\in \La}\big\|V_{g\otimes \bar{g}}K_A(\mu_1,\la_1,\mu_2,-\la_2)\big\|_{l^q_{\mu}(\La)}<\fy$.
	\en
	Moreover, if one of the above statements holds, we have the norm estimate 
	\be
    \begin{split}
        &\|A\|_{\calL(M^p(\rd), M^q(\rd))}
		\sim 
		\sup_{z\in \rdd}\|A(\pi(z)g)\|_{\mqd}
        \sim
        \sup_{\la\in \La}\|A(\pi(\la)g)\|_{\mqd}
		\\
		\sim &
		\sup_{z\in \rdd}\big\|V_{g\otimes \bar{g}}K_A(z_1',z_1,z_2',-z_2)\big\|_{L^q_{z'}(\rdd)}
  \sim 
  \sup_{\la\in \La}\big\|V_{g\otimes \bar{g}}K_A(\mu_1,\la_1,\mu_2,-\la_2)\big\|_{L^q_{\mu}(\La)}.
    \end{split}
	\ee
\end{theorem}
\begin{proof}
	The relation $(1)\Longrightarrow(2)\Longrightarrow(3)$ follows directly by the fact that
 \be
     \|\pi(z)g\|_{M^p}
     \sim \|V_{\phi}(\pi(z)g)\|_{L^p}
     = \|T_zV_{\phi}g\|_{L^p}
     = \|V_{\phi}g\|_{L^p}
     \sim \|g\|_{M^p}.
 \ee
The equivalent relations $(2)\Longleftrightarrow (4)$ and $(3)\Longleftrightarrow (5)$ follow by the facts
\be
V_g(A(\pi(z)g))(z')
=
\lan A\pi(z)g,\pi(z')g\ran=\lan K_A, \pi(z')g\otimes \overline{\pi(z)g}\ran=V_{g\otimes \bar{g}}K_A(z_1',z_1,z_2',-z_2)
\ee
and
\be
\|A(\pi(z)g)\|_{\mqd}
\sim \|V_g(A(\pi(z)g))\|_{L^q(\rdd)}.
\ee
In order to complete this proof, we only need to verify the relation $(3)\Longrightarrow (1)$.
  For $f\in\calS(\rd)$, we have the expansion
\begin{equation*}
    f=\sum_{\la \in \La}\langle f, \pi(\la)\g\rangle \pi(\la)g
\end{equation*}
with unconditional convergence in $\calS(\rd)$.
 From this and the fact that $A\in \calL(\calS(\rd), \calS'(\rd))$, we write
 \be
  Af=\sum_{\la \in \La}\langle f, \pi(\la)\g\rangle A(\pi(\la)g),
 \ee
 with unconditional convergence in $\calS'(\rd)$.
 By Proposition \ref{pp-Gabor-mod}, the desired conclusion follows by
 \be
\begin{split}
    \|Af\|_{\mqd}
    \leq &
    \big(\sum_{\la \in \La}|\langle f, \pi(\la)\g\rangle|^{q\wedge 1} \|A(\pi(\la)g)\|_{\mqd}^{q\wedge 1}\big)^{\frac{1}{q\wedge 1}}
    \\
    \leq &
    \sup_{\la\in \La}\|A(\pi(\la)g)\|_{\mqd}\big(\sum_{\la \in \La}|\langle f, \pi(\la)\g\rangle|^{q\wedge 1}\big)^{\frac{1}{q\wedge 1}}
    \\
    \leq &
    \sup_{\la\in \La}\|A(\pi(\la)g)\|_{\mqd}\big(\sum_{\la\in \La}|\lan f, \pi(\la)\g\ran|^p\big)^{1/p}
    \\
    \sim &
    \sup_{\la\in \La}\|A(\pi(\la)g)\|_{\mqd}\|f\|_{\mpd}.
\end{split}
 \ee
\end{proof}
By using a duality argument, we obtain the following conclusion.
\begin{theorem}\label{thm-BK-dual}
	Suppose that $p\in [1, \i]$.
    Let $g\in \calS(\rd)$ such that $\calG(g,\La)$ is a frame on $L^2(\rd)$ with the canonical dual frame $\calG(\g,\La)$.
	Let $A\in \calL(\calS(\rd), \calS'(\rd))$ be a linear operator with
	distributional kernel $K_A\in \calS'(\rdd)$.	
	The following statements are equivalent.
	\bn
	\item
	$A\in\calL(M^p(\rd), M^{\i}(\rd))$;
		\item 
	$\sup_{z\in \rdd}\big\|V_{g\otimes \bar{g}}K_A(z_1,z_1',z_2,-z_2')\big\|_{L^{p'}_{z'}(\rdd)}<\fy$;
 	\item 
	$\sup_{\la\in \La}\big\|V_{g\otimes \bar{g}}K_A(\la_1,\mu_1,\la_2,-\mu_2)\big\|_{l^{p'}_{\mu}(\La)}<\fy$.
	\en
	Moreover, if one of the above statements holds, we have the norm estimate 
	\be
 \begin{split}
     \|A\|_{\calL(M^p(\rd), M^{\i}(\rd))}
		\sim &
		\sup_{z\in \rdd}\big\|V_{g\otimes \bar{g}}K_A(z_1,z_1',z_2,-z_2')\big\|_{L^{p'}_{z'}(\rdd)}
  \\
  \sim &
  \sup_{\la\in \La}\big\|V_{g\otimes \bar{g}}K_A(\la_1,\mu_1,\la_2,-\mu_2)\big\|_{l^{p'}_{\mu}(\La)}.
 \end{split}
	\ee
\end{theorem}
\begin{remark}
	The (quasi-)norms of $(4),(5)$ in Theorem \ref{thm-BK} and $(2),(3)$ in Theorem \ref{thm-BK-dual} 
 are actually that of the so-called mixed modulation spaces in \cite{CorderoNicola2019JFAA}.
 Our theorems generalized the main conclusions in \cite{CorderoNicola2019JFAA}.
\end{remark}

Next, we consider the compactness kernel theorem of modulation space. 
With the help of Gabor characterization, we give the following characterization of totally boundedness subset in modulation space.
\begin{theorem}\label{thm-tbs}
Let $g\in \calS(\rd)$ such that $\calG(g,\La)$ is a frame on $L^2(\rd)$ with the canonical dual frame $\calG(\g,\La)$.
	Let $p\in (0,\fy)$, and $\calF$ be a subset of $\mpd$.
	The following three statements are equivalent: 
	\bn
	\item $\calF$ is totally bounded;
	\item $\calF$ satisfies the properties of uniformly boundedness and disappearance:
	 \bn
	 \item Uniformly boundedness: 
	 \ben\label{thm-tbs-cd1}
	 \sup_{f\in \calF}\|f\|_{\mpd}<\fy;
	 \een
	 \item Uniformly disapperance: 
	 \ben\label{thm-tbs-cd2}
	 \lim_{N\rightarrow \fy}\sup_{f\in \calF}\big\|V_gf(\la)\big\|_{l^p(\La\bs [-N,N]^{2d})}=0;
	 \een
	 \en
	\item $\wt{\calF}:=\Big\{\big\{V_gf(\la)\big\}_{\la\in \La}: f\in \calF\Big\}$ is totally bounded in $l^p(\La)$.
	\en
\end{theorem}
\begin{proof}
    	$(1)\Longrightarrow(2):$
    The statement $(a)$ follows directly by the fact that a totally bounded subset is bounded.
	The statement $(b)$ follows by the estimate
 \be
 \begin{split}
     \|V_gf(\la)\|_{l^p(\La\bs [-N,N]^{2d})}
     \lesssim &
     \|V_g(f-h_j)(\la)\|_{l^p(\La\bs [-N,N]^{2d})}+\|V_gh_j(\la)\|_{l^p(\La\bs [-N,N]^{2d})}
     \\
     \lesssim &
     \|f-h_j\|_{\mpd}+\|V_gh_j(\la)\|_{l^p(\La\bs [-N,N]^{2d})},
 \end{split}
 \ee
 where $\{h_j\}_{j=1}^N$ is an $\ep$-net of $\calF$.

 The relation $(2)\Longrightarrow(1)$ follows by the similar argument as in the proof of Theorem \ref{thm-tbsa}, by using the Gabor expansion
 of modulation space. The relation $(2)\Longleftrightarrow(3)$  follows by Proposition \ref{lm-lp-tot.bdd} and \ref{pp-Gabor-mod}.
\end{proof}

\begin{remark}\label{rmk-tbs}
    We point out that a characterization of totally boundedness subset in a family of more general function spaces, the so-called coorbit spaces, was established in \cite{DoerflerFeichtingerGroechenig2002CM}.
    As a special case, they give the corresponding conclusion of modulation space, 
    see \cite[Theorem 5]{DoerflerFeichtingerGroechenig2002CM}.
    By using this conclusion, the property (b) in statement (2) of Theorem \ref{thm-tbs} can be replaced by 
    	 \be
	 \lim_{N\rightarrow \fy}\sup_{f\in \calF}\big\|V_gf(z)\big\|_{L^p(\rdd\bs [-N,N]^{2d})}=0.
	 \ee
\end{remark}

Using Theorem \ref{thm-tbs} and the similar argument as in the proof of Theorem \ref{thm-CKa},
we give the corresponding compactness kernel theorem for modulation spaces.
\begin{theorem}[Compactness kernel theorem for modulation spaces]\label{thm-CK}
Let $g\in \calS(\rd)$ such that $\calG(g,\La)$ is a frame on $L^2(\rd)$ with the canonical dual frame $\calG(\g,\La)$.
	Suppose that $p,q\in (0,\fy)$ satisfying $p\leq q\wedge 1$.
	Let $A\in \calL(\calS(\rd), \calS'(\rd))$ be a linear operator with
	distributional kernel $K_A$.	
	The following statements are equivalent.
	\bn
	\item
	$A\in\calK(M^p(\rd), M^q(\rd))$;
	\item
	$\{A(\pi(z) g): z \in \rdd\}$
	is a totally bounded subset of $M^{q}(\rd)$;
	\item
	$\{A(\pi(\la) g): \la \in \La\}$
	is a totally bounded subset of $M^{q}(\rd)$;
	\item 
 The following two properties are valid:
	
	(a) Uniformly boundedness:
	\be
	\sup_{z\in \rdd}\big\|V_{g\otimes \bar{g}}K_A(z_1',z_1,z_2',-z_2)\big\|_{L^q_{z'}(\rdd)}<\fy,
	\ee
	(b) Uniformly disappearance:
	\be
	\lim_{N\rightarrow \fy}\sup_{z\in \rdd}\big\|V_{g\otimes \bar{g}}K_A(z_1',z_1,z_2',-z_2)\big\|_{L^q_{z'}(\rdd\bs [-N,N]^{2d})}=0
	\ee
	\item
	The following two properties are valid:
	
	(a) Uniformly boundedness:
	\be
	\sup_{\la\in \La}\big\|V_{g\otimes \bar{g}}K_A(\mu_1,\la_1,\mu_2,-\la_2)\big\|_{l^q_{\mu}(\rdd)}<\fy,
	\ee
	(b) Uniformly disappearance:
	\be
	\lim_{N\rightarrow \fy}\sup_{\la\in \La}\big\|V_{g\otimes \bar{g}}K_A(\mu_1,\la_1,\mu_2,-\la_2)\big\|_{l^q_{\mu}(\La\bs [-N,N]^{2d})}=0
	\ee
	\en
\end{theorem}
\begin{proof}
	The relation $(1)\Longrightarrow(2)\Longrightarrow(3)$ follows by $A\in \calK(\mpad,\mqad)$
	and the fact that both $\{\pi(z)g\}_{z\in \rdd}$ and $\{\pi(\la)g\}_{\la\in \La}$ are bounded subsets of modulation space.
The equivalent relations $(2)\Longleftrightarrow (4)$ and $(3)\Longleftrightarrow (5)$ follow by 
Remark \ref{rmk-tbs} and Theorem \ref{thm-tbs} respectively.
Finally, we turn to the proof of $(3)\Longrightarrow (1)$.
	Let $\calF$ be bounded subset of $\mpd$. 
 From $(3)$, we have $\{A(\pi(\la) g): \la \in \La\}$ is a bounded subset of $M^{q}(\rd)$. 
 Then by Theorem \ref{thm-BK-dual}, we obtain $A\in\calL(\mpd, \mqd)$.
 For $f\in\calF$, we write the Gabor expansion by
	\be
	f=\sum_{\la \in \La}\langle f, \pi(\la)\g\rangle \pi(\la)g
	\ee
	with unconditional convergence in $\mpd$.
	Applying the operator $A$ on both sides of the above expansion, we obtain
		\be
		\begin{split}
			Af
			=\sum_{\la \in \La}\langle f, (\pi(\la)\g\rangle A(\pi(\la)g),
		\end{split}
	\ee
	with unconditional convergence in $\mqd$.
	Note that 
	\be
	\begin{split}
		\big\|\{\langle f, (\pi(\la)\g\rangle\}_{\la\in \La}\big\|_{l^{q\wedge 1}(\La)}
		\lesssim 
		\big\|\{\langle f, (\pi(\la)\g\rangle\}_{\la\in \La}\big\|_{l^{p}(\La)}
		\sim  \|f\|_{\mpd}\leq \sup_{f\in \calF}\|f\|_{\mpd}.
	\end{split}
	\ee
	Combining this with the fact that $\{A(\pi(\la) g): \la \in \La\}$
	is a totally bounded subset of $\mqd$, we use Proposition \ref{pp-tbs} to deduce that 
	$
	\{Af: f\in \calF\}
	$
	is a totally bounded subset of $\mqd$. This implies the desired conclusion that $A\in \calK(\mpd,\mqd)$.
\end{proof}

\begin{remark}
    In the proof of our main theorems, a crucial step is the expansion of $f\in M^p_{\al}$. 
    To this end, we have established some specific decompositions applicable to our proof, such as Propositions \ref{pp-Fourierexp-alpha}
    and \ref{pp-exp-alpha}. In fact, such decompositions are closely related to the so-called Banach frames.
    One can utilize Proposition \ref{pp-exp-alpha} and its proof to construct a Banach frame for $\al$-modulation space.
    See also \cite{Fornasier2007AaCHA} and \cite{BorupNielsen2006JMAA} for the construction of Banach frame
    for $\al$-modulation spaces.
    Finally, we point out that by using an appropriate Banach frame for $\al$-modulation space, one should be able to follow the approach of modulation space to derive the kernel theorems for $\al$-modulation space. 
    However, such an approach hides the essence that the action of an operator on a function space can be transformed into its action on atoms.
\end{remark}

\subsection*{Acknowledgements} Supported by the National Natural Science Foundation of China [12371100] and Fujian Provincial Natural Science Foundation of China [2022J011241].

\bibliographystyle{abbrv}

\end{document}